\documentclass[12pt]{amsart}

\usepackage{amsmath,amssymb,amsfonts,latexsym}
\usepackage{mathrsfs}
\usepackage[dvips]{geometry}
\usepackage{array}

\usepackage[vcentermath]{youngtab}

\DeclareFontFamily{T1}{calligra}{}
\DeclareFontShape{T1}{calligra}{m}{n} {<-> callig15}{}

\newtheorem{thm}{Theorem}[section]
\newtheorem{prop}[thm]{Proposition}
\newtheorem{cor}[thm]{Corollary}
\newtheorem{lem}[thm]{Lemma}

\newtheorem{con}[thm]{Conjecture}
\newtheorem{rem}[thm]{Remark}

\renewenvironment{proof}{\noindent{\em Proof.}}{\hfill $\Box$}

\numberwithin{equation}{section}

\newcommand{\bS}{\mathbf{S}}

\newcommand{\codim}{{\rm codim\,}}

\newcommand{\ov}{\overline}

\newcommand{\B}{\mathcal{B}}

\newcommand{\Z}{\mathcal{Z}}
\newcommand{\X}{\mathcal{X}}

\newcommand{\F}{\mathcal{F}}

\newcommand{\Sl}{\mathfrak{s}\mathfrak{l}}

\newcommand{\nil}{\mathfrak{n}}

\newcommand{\rar}{\rightarrow}

\begin{document}

\title[Singular components of Springer fibers in the two-column case]{Some characterizations of singular components of Springer fibers in the two-column
case}

\author{Lucas Fresse}
\address{Department of Mathematics, the Weizmann Institute of Science, Rehovot 76100, Israel}
\email{lucas.fresse@weizmann.ac.il}
\author{Anna Melnikov}
\address{Department of Mathematics, University of Haifa, Haifa 31905, Israel}
\email{melnikov@math.haifa.ac.il}

\begin{abstract}
Let $u$ be a nilpotent endomorphism of a finite dimensional $\mathbb{C}$-vector space.
The set ${\mathcal F}_u$ of $u$-stable complete flags is a projective algebraic variety called a Springer fiber.
Its irreducible components are parameterized by a set of standard tableaux.
We provide three characterizations of the singular components of ${\mathcal F}_u$ in the case $u^2=0$.
First, we give the combinatorial description of standard tableaux corresponding to singular components.
Second, we prove that a component is singular if and only if its Poincar\'e polynomial is not palindromic.
Third, we
show that a component is singular when it has too many intersections of codimension one
with other components.
Finally, relying on the second criterion,
we infer that, for $u$ general, whenever ${\mathcal F}_u$ has a singular component,
it admits a component whose Poincar\'e polynomial is not palindromic.
This work relies on a previous criterion of singularity for components of ${\mathcal F}_u$ in the case $u^2=0$
by the first author and on
the description of the $B$-orbit decomposition of orbital varieties of nilpotent order two by the second author.



\end{abstract}

\maketitle

\section{Introduction}
\subsection{Springer fibers and singularity of their irreducible components}
\label{introduction-1}
Let $V$ be a $\mathbb{C}$-vector space of dimension $n\geq 0$
and let $u:V\rightarrow V$ be a nilpotent endomorphism.
Let ${\mathcal F}$ be the set of complete flags, i.e. maximal chains of subspaces
$(0=V_0\subset V_1\subset ...\subset V_n=V)$. The set ${\mathcal F}$ is a
projective algebraic variety, called the {\em flag variety}.
Let ${\mathcal F}_u$ be the subset of $u$-stable complete flags, i.e. flags
$(V_0,...,V_n)$ such that $u(V_i)\subset V_i$ for every $i$.
The set ${\mathcal F}_u$ is a closed subvariety of ${\mathcal F}$. It is called
{\em Springer fiber}, since it can be identified with the fiber over $u$ of the Springer resolution
(see \cite{Springer1}, \cite{Springer2}).

Obviously, the variety ${\mathcal F}_u$ depends only on the Jordan form of $u.$
It is reducible and singular unless $u$ is zero or regular,
but the irreducible components of ${\mathcal F}_u$ can be singular or smooth.
Up to now, the problem to determine, for a given $u$, all the singular components of ${\mathcal F}_u$
is  solved only  in a few special cases.
All the components of ${\mathcal F}_u$ are smooth in three cases:
if $u$ has only one nontrivial Jordan block (Vargas, cf. \cite{V}), if $u$ has only two Jordan blocks
(Fung, cf. \cite{Fung}), if $u$ has two nontrivial and one trivial Jordan blocks (Fresse-Melnikov, cf. \cite{F-M}).
In all other case (for $n>6$) ${\mathcal F}_u$ has some singular components. However only in the case $u^2=0$
a necessary and sufficient condition
of singularity for components is given (cf. \cite{F}, or \ref{previous-criterion} below).
In this article, we mainly concentrate on the case $u^2=0$,
for which we give three new characterizations of
the singular components of ${\mathcal F}_u$.

\subsection{Parametrization of the irreducible components of ${\mathcal F}_u$ by standard Young tableaux}
\label{parameterization-components}
Let $\lambda(u)=(\lambda_1\geq...\geq \lambda_r)$ be the sizes of the Jordan blocks
of $u$, and let the Young diagram $Y(u)$, or $Y_\lambda$,
 be an array of $r$  rows
of boxes starting on the left, with the $i$-th row containing
$\lambda_i$ boxes.
Since $(\lambda_1,...,\lambda_r)$ is a partition of $\mathrm{dim}\,V=n$, the Young diagram $Y(u)$ contains $n$ boxes.
Let $\lambda(u)^*=(\mu_1,...,\mu_s)$ denote the  conjugate partition, that is the list of the lengths
of the columns in $Y(u)$.
By \cite[\S II.5.5]{Spaltenstein}, the dimension of ${\mathcal F}_u$
is given by the formula
\begin{eqnarray}\label{eq1}
\mathrm{dim}\,{\mathcal F}_u=\sum_{i=1}^s\frac{\mu_i(\mu_i-1)}{2}.
\end{eqnarray}
In the case $u^2=0$, which we consider in this article, the diagram $Y(u)$ has two columns,
thus $\mathrm{dim}\,{\mathcal F}_u=\frac{1}{2}\mu_1(\mu_1-1)+\frac{1}{2}\mu_2(\mu_2-1)$.

Given a Young diagram $Y(u)$, fill in its boxes with the numbers
$1,\ldots,n$ in such a way that the entries increase in rows from
left to right and in columns from top to bottom. Such an array is
called  a {\em standard Young tableau} or simply a {\em standard tableau} of shape $Y(u)$.


Following \cite{Spaltenstein}, we introduce a parametrization
of the components of ${\mathcal F}_u$ by standard tableaux of shape $Y(u)$.
For a standard tableau $T$, for $i=1,...,n$, let $T_i$ be the subtableau of $T$
containing the entries $1,...,i$ and let $Y_i^T\subset Y(u)$ be the shape of $T_i.$
Let $F=(V_0,...,V_n)\in{\mathcal F}_u$, then for $i=1,...,n$,
the subspace $V_i$ is $u$-stable, thus, we can consider the restriction
map $u_{|V_i}:V_i\rightarrow V_i$ which is a nilpotent endomorphism.
Its Jordan form is represented by a Young diagram $Y(u_{|V_i})$,
which is a subdiagram of $Y(u)$.
Set
\[{\mathcal F}_u^T=\{(V_0,...,V_n)\in{\mathcal F}_u:Y(u_{|V_i})=Y_i^T\ \ \forall i=1,...,n\}.\]
We get a partition ${\mathcal F}_u=\bigsqcup_{T}{\mathcal F}_u^T$ parameterized
by the standard tableaux of shape $Y(u)$. By \cite[\S II.5.4--5]{Spaltenstein},
for each $T$, the set ${\mathcal F}_u^T$ is a locally closed, irreducible
subset of ${\mathcal F}_u$, and $\mathrm{dim}\,{\mathcal F}_u^T=\mathrm{dim}\,{\mathcal F}_u$.
Define ${\mathcal K}^T=\overline{{\mathcal F}_u^T}$, to be  the  closure in Zariski topology.
Then ${\mathcal K}^T$, for $T$ running over the set of standard tableaux
of shape $Y(u)$, are the irreducible components of ${\mathcal F}_u$.
Moreover, we have $\mathrm{dim}\,{\mathcal K}^T=\mathrm{dim}\,{\mathcal F}_u$ for every $T$.

\medskip

In the sequel, we suppose that $Y(u)$ has two columns of lengths $(n-k,k)$.

\subsection{A basic criterion for the singularity of a component}
\label{previous-criterion}
We recall from \cite{F} a necessary and sufficient condition of singularity
for the components of ${\mathcal F}_u$.

We call {\em row-standard tableau} an array obtained from $Y(u)$
by filling in its boxes with the numbers $1,\ldots,n$ in such a way that the entries increase in rows from left to right.
We associate a flag $F_\tau$ to each row-standard tableau $\tau$.
To do this, we fix a Jordan basis $(e_1,...,e_n)$ of $u$, such that $u(e_i)=0$ for $i=1,...,n-k$
and $u(e_i)=e_{i-n+k}$ for $i=n-k+1,...,n$.
We rely on the particular tableau $\tau_0$ of shape $Y(u)$ numbered from top to bottom
with the entries $1,...,n-k$ in the first column and the entries $n-k+1,...,n$ in the second column.
For a permutation $w\in\mathbf{S}_n$, let $w\tau_0$ be the tableau obtained from $\tau_0$ after replacing
each entry $i$ by $w(i)$. For $\tau$ row-standard, let $w_\tau\in\mathbf{S}_n$ be such
that $\tau=w_\tau^{-1}\tau_0$. Then let $F_\tau=(V_0,\ldots,V_n)$ be the flag defined by
$V_i=\langle e_{w_\tau(1)},\ldots,e_{w_\tau(i)}\rangle$. Thus $F_\tau\in{\mathcal F}_u$.
For $T$ standard, a characterization of row-standard tableaux $\tau$ such that $F_\tau\in{\mathcal K}^T$
is given in \cite[\S 2.3]{F}.

Let $X(\tau_0)$ denote the set of row-standard tableaux obtained from
$\tau_0$ by interchanging two entries $i,j$ with $i\leq n-k$.
By  \cite[Theorem 3.1]{F} one has

\begin{thm}\label{previous-crit}
Suppose that $Y(u)$ has two columns of lengths $(n-k,k)$.
Let $T$ be a standard tableau of shape $Y(u)$.
The component ${\mathcal K}^T$ is singular if and only if $|\{\tau\in X(\tau_0):F_\tau\in{\mathcal K}^T\}|>\frac{1}{2}(n-k)(n-k-1)$.
\end{thm}

Our  three new characterizations of singular components of ${\mathcal F}_u$ are based on this criterion.
The first one, purely combinatorial, simplifies drastically  the criterion above,
the other two involve remarkable properties.

\subsection{Combinatorial criterion of singularity}
\label{first-criterion}
Let $\mathbf{S}_n^2$ be the set of involutions in the symmetric group $\mathbf S_n$, that is  $\bS_n^2=\{\sigma\in\mathbf{S}_n\ :\ \sigma^2=I\}$,
and let $\mathbf{S}_n^2(k)\subset\mathbf{S}_n^2$ be the subset of permutations
which are product of $k$ pairwise disjoint transpositions, that is any $\sigma\in\mathbf{S}_n^2(k)$ can be written
(in a cyclic form) as
$\sigma=(i_1,j_1)\cdots (i_k,j_k)$ where $i_p<j_p$ for every $p=1,...,k$ and $\sigma(i_p)=j_p$.
Moreover this factorization is unique
up to the order of the factors. For $i\ :\ 1\leq i\leq n$ we call $i$ an end point of $\sigma$ if $\sigma(i)\ne i$,
   we call $i$ a fixed point of $\sigma$ if  $\sigma(i)=i$.
For $i,j\ :\ 1\leq i<j\leq n$ we write $(i,j)\in \sigma$ if $\sigma(i)=j$ and we write $(i,j)\notin \sigma$ otherwise.

Let $T$ be a standard tableau of shape $Y(u)$. We associate
the involution $\sigma_T\in \mathbf{S}_n^2(k)$ to $T$ by the following procedure.
Let $a_1<...<a_{n-k}$ (resp. $j_1<...<j_k$) be the entries in the first
(resp. second) column of $T$.
Put $\sigma_T=(i_1,j_1)\cdots(i_k,j_k)$ where $i_1=j_1-1$
and $i_p=\mathrm{max}\{a\in\{a_1,...,a_{n-k}\}\setminus \{i_1,...,i_{p-1}\}:a<j_p\}$
for $p=2,...,k$.

For $i=1,...,n$, let $c_T(i)\in\{1,2\}$ be the index of the column of  $T$
containing $i$. Write $\tau^*(T)=\{i\in\{1,...,n-1\}\ :\ c_T(i)<c_T(i+1)\}$.
Let $|\tau^*(T)|$ be the cardinality of $\tau^*(T)$.

\medskip \noindent
{\em Example.} Let $T=\mbox{\scriptsize $\young(14,26,37,5,8)$}$. Then $\sigma_T=(3,4)(5,6)(2,7)$.
Thus $2,3,4,5,6,7$ are the end points of $\sigma_T$, and $1,8$ are the fixed points.
We have $\tau^*(T)=\{3,5\}$.

\medskip
Our first criterion gives an explicit description of tableaux $T$
for which the component ${\mathcal K}^T$ is singular.

\begin{thm}\label{first-crit} Suppose that $Y(u)$ has two columns. Let $T$ be a standard tableau of shape $Y(u)$.
Let ${\mathcal K}^T\subset {\mathcal F}_u$ be the irreducible component
associated to $T$. \\
(a) If $|\tau^*(T)|=1$, then ${\mathcal K}^T$ is smooth. \\
(b) If $|\tau^*(T)|=2$, then ${\mathcal K}^T$ is smooth if and only if at least one of
$\{1,n\}$ is an end point of $\sigma_T$. \\
(c) If $|\tau^*(T)|=3$, then ${\mathcal K}^T$ is smooth if and only if both
$1$ and $n$ are end points of $\sigma_T$ and $(1,n)\notin\sigma_T$. \\
(d) If $|\tau^*(T)|\geq 4$, then ${\mathcal K}^T$ is singular.
\end{thm}

For example, for $T$  above, we have $|\tau^*(T)|=2$ and both $1,8$
are not end points of $\sigma_T$, hence the component ${\mathcal K}^T$ is singular.

\subsection{Singularity and Betti numbers distribution}
\label{second-criterion}
We characterize the singular components by their cohomology.
We consider the classical sheaf cohomology with rational coefficients.
Let $X$ be an algebraic variety of dimension $d_X$.
Let $H^m(X,\mathbb{Q})$ denote the cohomology space in degree $m$.
We have $H^m(X,\mathbb{Q})=0$ unless $m\in\{0,...,2d_X\}$.
The numbers $\{\mathrm{dim}\,H^m(X,\mathbb{Q})\}_{m=0}^{2d_X}$ are called Betti numbers.
The distribution of Betti numbers is called  symmetric if
$\mathrm{dim}\,H^m(X,\mathbb{Q})=\mathrm{dim}\,H^{2d_X-m}(X,\mathbb{Q})$
for any $ m=0,...,d_X$.
Equivalently, we say that the Poincar\'e polynomial $P_X(x):=\sum_{m=0}^{2d_X}H^{2d_X-m}(X,\mathbb{Q})\,x^m$ is palindromic.
If $X$ is irreducible, smooth and projective,
then, by Poincar\'e duality,  the distribution of Betti numbers
of $X$ is symmetric.
In particular, the distribution of Betti numbers of a smooth component of ${\mathcal F}_u$
is symmetric.
Our second criterion is the following.

\begin{thm}\label{second-crit}
Suppose that $Y(u)$ has two columns. Let $T$ be a standard tableau of shape $Y(u)$.
The component ${\mathcal K}^T\subset {\mathcal F}_u$ is smooth
if and only if the Poincar\'e polynomial of ${\mathcal K}^T$ is palindromic.
\end{thm}

We compute Betti numbers via the construction of a cell decomposition of ${\mathcal K}^T$ in Section \ref{section-cell-decomposition}.

We would like to understand the distribution of Betti numbers for singular components outside of the two-column case.
For $n\leq6$ all the components outside of two-column case are smooth. For $n=7$, as it is shown in \cite{F-M}, there is a unique singular component of $\F_u$ outside of two-column case, and $Y(u)=(3,2,2)$ in this case.
In Section \ref{further-speculations} we  check that the distribution of Betti numbers for this component is non-symmetric as well.
Moreover, using this fact and the above theorem for the two-column case we show in Section \ref{further-speculations} that, if $\F_u$ has singular components, then it has at least one singular component with non-symmetric distribution of Betti
numbers.

This brings us to the conjecture that it may be a general phenomenon, namely

\begin{con}
Let $T$ be a standard tableau of shape $Y(u)$.
The component ${\mathcal K}^T\subset {\mathcal F}_u$ is smooth
if and only if the Poincar\'e polynomial of ${\mathcal K}^T$ is palindromic.
\end{con}

Recall that, among classical varieties, Schubert varieties share this property
that their singularity is characterized by the non-symmetry of the distribution
of Betti numbers (proved by Carrell and Peterson \cite{Carrell2}).

Note also that an algebraic variety for which the Poincar\'e duality fails is in particular rationally singular
(see for example \cite[Proposition 6.19]{Carrell}). Then,
due to the above, we obtain that the description of Springer fibers admitting rationally singular
irreducible components coincides with the description of those admitting singular components.

The fact that the singular components of $\F_u$ for $u^2=0$ are rationally singular
(and also the fact that they are normal) is shown
 by Perrin and Smirnov \cite{Smirnov} using different methods.

\subsection{Outline of the proof of Theorems \ref{first-crit} and \ref{second-crit}}
\label{plan}
Before stating our third criterion, let us describe our plan.

Let $Z(u)=\{g\in GL(V):gu=ug\}$ be the stabilizer of $u$ in $GL(V).$
It is  a closed, connected subgroup of $GL(V)$. Its natural action on flags
leaves ${\mathcal F}_u$ and every component of ${\mathcal F}_u$ invariant.
As a preliminary step, assuming that $Y(u)$ has two columns of lengths $(n-k,k)$,
we show that
${\mathcal F}_u$ is a finite union of $Z(u)$-orbits,
parameterized by the permutations $\sigma\in \mathbf{S}_n^2(k)$,
and we also describe the decomposition into $Z(u)$-orbits
of each component ${\mathcal K}^T\subset{\mathcal F}_u$
(Section \ref{section-Z(u)-orbites}).
In particular, we show that ${\mathcal F}_u$ has a unique $Z(u)$-orbit of minimal dimension
$d_0$.
We derive this description of $Z(u)$-orbits of components of ${\mathcal F}_u$
from the description of $B$-orbits of orbital varieties given in \cite{Ml-p}.

Next, we show that each $Z(u)$-orbit of ${\mathcal F}_u$ admits
a cell decomposition, such that the number of cells in the decomposition
and their codimensions are the same for all the $Z(u)$-orbits (cf. Proposition \ref{prop-cell-decomposition}).
We deduce that the distribution of Betti numbers of a component ${\mathcal K}^T$
is symmetric if and only if the number of $Z(u)$-orbits of ${\mathcal K}^T$
of codimension $m$ is equal to the number of $Z(u)$-orbits of ${\mathcal K}^T$
of dimension $d_0+m$, for every $m$.

Set a standard tableau $T$ to be of first type if
$|\tau^*(T)|=1$, or $|\tau^*(T)|=2$ and $1$ or $n$ is an end point of $\sigma_T$,
or $|\tau^*(T)|=3$ and $1,n$ are end points of $\sigma_T$ and $(1,n)\notin \sigma_T$,
and set $T$ to be of second type otherwise.
In Section \ref{section-on-nonsingular-components} we prove that, if $T$ is of first type, then
the component ${\mathcal K}^T$ is smooth, by using Theorem \ref{previous-crit}
and an inductive argument.
Set
$$\overline{k}=\left\{\begin{array}{ll}
k &{\rm if}\ k=\frac{n}{2},\\
k+1 & {\rm otherwise}.\\
\end{array}\right. $$
In Section \ref{section-on-singular-components}, we show that ${\mathcal F}_u$ has exactly $\overline{k}$ $Z(u)$-orbits
of dimension $d_0+1$, whereas, if $T$ is of second type, then
${\mathcal K}^T$ contains more than $\overline{k}$ $Z(u)$-orbits of codimension 1.
This
proves Theorem \ref{first-crit} and Theorem \ref{second-crit}.

\subsection{ Singularity and intersections of codimension 1}
\label{third-criterion}
By the way, we obtain the third criterion of singularity.
For  a standard tableau $T$, let $\eta(T)$ be the number
of components ${\mathcal K}'\subset{\mathcal F}_u$ such
that $\mathrm{codim}\,{\mathcal K}^T\cap {\mathcal K}'=1$.
We have the following

\begin{thm}\label{third-crit}
Suppose that $Y(u)$ has two columns of lengths $(n-k,k)$. Let $T$ be a standard tableau of shape $Y(u)$.
The component ${\mathcal K}^T\subset {\mathcal F}_u$ is singular
if and only if $\eta(T)>\overline{k}$.
\end{thm}

The proof is given in \ref{proof-third-criterion}.
This theorem connects the question of singularity to
the classical question of components of Springer fibers intersecting
in codimension 1. Consider the graph whose vertices are the standard tableaux of shape $Y(u)$,
with an edge $(T,T')$ if $\mathrm{codim}\,{\mathcal K}^T\cap{\mathcal K}^{T'}=1$.
It is stated in \cite[Conjecture 6.3]{K-L} that this graph is a $W$-graph,
which defines an irreducible representation of the Hecke algebra of the symmetric group $\mathbf{S}_n$.
The conjecture is known to be true in the two-column case (see \cite{Ml-p}).
Then the theorem shows that the singular components of ${\mathcal F}_u$
correspond to the vertices of the graph belonging to more than $\overline{k}$ edges.

\subsection{Notation}
In what follows, $V$ denotes an $n$-dimensional $\mathbb{C}$-vector space, the subspaces are denoted by
$W,W',...$, the flags are written as $(V_0,...,V_n)$ or $(V_0\subset...\subset V_n)$ or $F$.
All along this article but Section \ref{further-speculations}, we consider a nilpotent endomorphism $u$ of nilpotent order two and of  rank $k$.
Its Young diagram $Y(u)$ consists of  two columns of lengths $n-k$ and $k$.
The standard tableaux of shape $Y(u)$ are usually denoted by $T,T',...$ and the corresponding
components in the Springer fiber ${\mathcal F}_u$ are denoted by ${\mathcal K}^T,{\mathcal K}^{T'},...$.
Permutations are usually written $\sigma,\sigma',...$, or $w,w',...$
We denote by $|A|$ the cardinality of a set $A$. For an algebraic set $U$, its closure (in Zariski topology)
is denoted by $\ov{U}.$

The reader can find the index of notation at the end of the paper.

\section{Decomposition of the variety ${\mathcal F}_u$ into $Z(u)$-orbits}

\label{section-Z(u)-orbites}

In this section, we show that the $Z(u)$-orbits of ${\mathcal F}_u$ are parameterized by the involutions
$\sigma\in\mathbf{S}_n^2(k)$. We denote  the orbit associated to $\sigma$ by ${\mathcal Z}_\sigma$.
We give the dimension formula for  ${\mathcal Z}_\sigma$ and describe its closure.
We also describe a component ${\mathcal K}^T\subset {\mathcal F}_u$ as the closure
of a maximal orbit ${\mathcal Z}_{\sigma_T}$ associated to $\sigma_T\in\mathbf{S}_n^2(k)$ (see \ref{first-criterion}).
Finally, we show that there is a unique minimal orbit ${\mathcal Z}_{\sigma_0}\subset{\mathcal F}_u$
associated to an involution $\sigma_0=\sigma_0(k)\in\mathbf{S}_n^2(k)$.

\subsection{Combinatorial set up}
\label{definition-link-pattern}
Let us begin with some combinatorial definitions.
Recall that $\mathbf{S}_n^2$ denotes the set of permutations $\sigma\in\mathbf{S}_n$ with $\sigma^2=I$,
and $\mathbf{S}_n^2(k)\subset\mathbf{S}_n^2$ denotes the subset of permutations
obtained as product of $k$ pairwise disjoint transpositions.
Following \cite{Ml-p}, the link pattern $P_\sigma$ corresponding to $\sigma\in\mathbf{S}_n^2$
is an array of $n$ points on a horizontal line where points $i<j$ are connected by an arc $(i,j)$
if $\sigma(i)=j$. Such points are called end points of an arc. We will not distinguish between $(i,j)\in \sigma$ and an arc $(i,j)$ of $P_\sigma .$
A point $p$ with $\sigma(p)=p$ is called a fixed point of $P_\sigma$. For $1\leq i\leq j\leq n$ let $[i,j]=\{i,i+1,\ldots,j\}$ denote the set of integer points of the interval $[i,j]$.
Set $P_\sigma^0=\{i\in[1,n]\ :\ \sigma(i)=i\}$
to be the set of fixed points,
 $P_\sigma^-=\{i\in[1,n]\ :\ i<\sigma(i)\}$ to be the set of left end points,
and  $P_\sigma^+=\{i\in[1,n]\ :\ i>\sigma(i)\}$ to be the set of right end points of $P_\sigma.$
The number of arcs in $P_\sigma$ is equal to $k$ whenever $\sigma\in\mathbf{S}_n^2(k)$.
For example,
for $\sigma=(1,3)(2,6)(4,7)\in \mathbf{S}_7^2$ one has
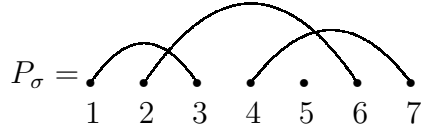
\begin{center}
\begin{picture}(100,65)(0,20)
\put(-20,40){$P_\sigma=$}
\multiput(10,40)(20,0){7}%
{\circle*{3}}
 \put(8,25){1}
 \put(28,25){2}
 \put(48,25){3}
 \put(68,25){4}
 \put(88,25){5}
 \put(108,25){6}
 \put(128,25){7}
 \qbezier(10,40)(30,70)(50,40)
 \qbezier(30,40)(70,100)(110,40)
 \qbezier(70,40)(100,80)(130,40)
\end{picture}
\end{center}
In our example $P_\sigma^0=\{5\},$\ $P_\sigma^-=\{1,2,4\}$ and $P_\sigma^+=\{3,6,7\}.$

Two arcs $(i,j)$, $(i',j')$ with $i,i'\in P_\sigma^-$, $i<i'$
are said to have a {\em crossing} if $i<i'<j<j'$.
We denote by $c(\sigma)$ the number of crossings of arcs in $P_\sigma$.
Let $p\in P_\sigma^0$ be a fixed point, we call an arc $(i,j)\in \sigma$ a {\em bridge} over $p$
if $p\in[i,j]$.
We denote by $b_p(\sigma)$ the number of bridges over $p$,
and we set $b(\sigma)=\sum_{p\in P_\sigma^0}b_p(\sigma)$.

In the example above, $c(\sigma)=2$ and $b(\sigma)=b_5(\sigma)=2.$

Note that $\sigma_T$ from \ref{first-criterion} is defined in such a way that $P_{\sigma_T}$
is a link pattern with $P_{\sigma_T}^+$ equal to the set of entries of the second column of $T$, without crossing
arcs and bridges over fixed points. For $T$, such a link pattern exists and is unique.

For integers $a,b\ :\ 1\leq a\leq b\leq n$,
we write $(i,j)\in[a,b]$ if $i,j\in[a,b]$. Let
$R_{a,b}(\sigma)$
be the number of arcs $(i,j)\in \sigma$
such that $(i,j)\in[a,b]$.
We define a partial order on $\bS_n^2$ by putting
$\sigma'\preceq \sigma$ if $R_{a,b}({\sigma'})\leq R_{a,b}(\sigma)$ for any $1\leq a\leq b\leq n$.

\subsection{Description of the $Z(u)$-orbit decomposition of ${\mathcal F}_u$}
\label{sect-proposition-orbits}
Let $\sigma\in\mathbf{S}_n^2(k)$.
A basis $(e_1,...,e_n)$ of $V$
is said to be a $\sigma$-basis
if $u(e_i)=0$ for $i\in P_\sigma^0\cup P_\sigma^-$ and $u(e_i)=e_{\sigma_i}$ for $i\in P_\sigma^+$.
Let ${\mathcal Z}_\sigma$ be the set of flags $F=(\langle e_1,...,e_i \rangle)_{i=0,...,n}$,
for some $\sigma$-basis $(e_1,...,e_n)$.
We have ${\mathcal Z}_\sigma\subset {\mathcal F}_u$.

For two $\sigma$-bases $(e_1,...,e_n)$ and $(e'_1,...,e'_n)$,
the element $g\in GL(V)$ defined by $g(e_i)=e'_i$ for any $i$,
satisfies $gu=ug$. It follows that two flags $F,F'\in {\mathcal Z}_\sigma$
belong to the same $Z(u)$-orbit.
Conversely, if $(e_1,...,e_n)$ is a $\sigma$-basis,
then $(ge_1,...,ge_n)$ is a $\sigma$-basis for any $g\in Z(u)$.
It follows that
for $F\in {\mathcal Z}_\sigma$ and $g\in Z(u)$ one has $gF\in {\mathcal Z}_\sigma$.
Thus ${\mathcal Z}_\sigma$ is a $Z(u)$-orbit.
We prove the following

\begin{prop}\label{proposition-orbits}
(a) The map $\sigma\mapsto{\mathcal Z}_\sigma$ is a bijection between $\mathbf{S}_n^2(k)$ and
the set of $Z(u)$-orbits of ${\mathcal F}_u$. \\
(b)  $\mathrm{dim}\,{\mathcal Z}_\sigma=\mathrm{dim}\,{\mathcal F}_u-c(\sigma)-b(\sigma)$. \\
(c) $\ov{\mathcal Z}_\sigma=\bigsqcup\limits_{\sigma'\in\bS_n^2(k)\, :\,\sigma'\preceq\sigma}{\mathcal Z}_{\sigma'}$. \\
(d) For $T$ a standard tableau of shape $Y(u)$, the component ${\mathcal K}^T\subset {\mathcal F}_u$
is the closure of the $Z(u)$-orbit ${\mathcal Z}_{\sigma_T}$.
\end{prop}

This description is the translation of the description  of $B$-orbits in the variety
of upper nilpotent matrices of nilpotent order $2$, given in \cite{Ml-p}, into the language of flags.

\medskip
\noindent
{\em Proof of Proposition \ref{proposition-orbits}.}

\smallskip
(1) First,
we recall the description of $B$-orbits.
Let $\nil=\nil_n$ be the algebra of strictly upper triangular $n\times n$ matrices
and put $\X=\{N\in \nil\ :\ N^2=0\}$ to be the subvariety of elements of nilpotent order two.
Put $\X(k)=\{N\in \X:\mathrm{rank}\,N=k\}$.
Let $B=B_n$ be the (Borel) group of upper triangular invertible $n\times n$ matrices and let $B$ act on $\nil$
by conjugation. This action stabilizes $\X$ and $\X(k)$.
For $\sigma\in\bS_n$
put
$$(N_\sigma)_{i,j}=\left\{\begin{array}{ll}
1&{\rm if}\ i<j\ {\rm and}\ \sigma(i)=j,\\
0&{\rm otherwise.}\\
\end{array}\right.
$$
Obviously for $\sigma\in\bS_n^2$ one has $N_\sigma^2=0$, hence $N_\sigma\in\X$,
and for $\sigma\in\bS_n^2(k)$ one has $N_\sigma\in\X(k)$.
Let ${\mathcal B}_\sigma=\{bN_\sigma b^{-1}:b\in B\}$ denote the $B$-orbit of $N_\sigma$.
We have the following description for $B$-orbits of $\X(k)$
(see \cite[\S 2.2]{Mx2} and \cite[\S 2.4 and \S 3.1]{Ml-p}): \\
(a) The map $\sigma\mapsto{\mathcal B}_\sigma$ is a bijective correspondence
between permutations $\sigma\in\mathbf{S}_n^2(k)$ and $B$-orbits of $\X(k)$. \\
(b) $\mathrm{codim}_{\X(k)}\,{\mathcal B}_\sigma=c(\sigma)+b(\sigma)$. \\
(c) $\ov{\mathcal B}_\sigma=\bigsqcup\limits_{\sigma'\preceq \sigma}{\mathcal B}_{\sigma'}$.

\smallskip
(2) Let us show a correspondence between $B$-orbits of ${\mathcal X}(k)$
and $Z(u)$-orbits of ${\mathcal F}_u$.
We fix a basis $\underline{e}=(e_1,...,e_n)$ of $V$ such that $u(e_i)=0$ for $i=1,...,n-k$
and $u(e_i)=e_{i-n+k}$ for $i=n-k+1,...,n$.
We identify $\psi\in{\rm End}(V)$ with its representing matrix with respect to the basis $\underline{e}$.
Let $F_0=(V_0,...,V_n)$ be the flag defined by $V_i=\langle e_1,...,e_i\rangle$.
Let $G(u)=\{g\in GL(V):g^{-1}ug\in\nil\}$. Note that $Z(u),B\subset G(u).$
Moreover, $Z(u)$ acts on $G(u)$ by left multiplication, that is $ zg\in G(u)$ for any $z\in Z(u),\ g\in G(u)$.
Since for $x\in \nil$ and $b\in B$ one has $b^{-1}xb\in \nil$, we get that $B$ acts on $G(u)$ by right multiplication, that is $gb\in G(u)$ for any $b\in B$, $g\in G(u).$

The map $\varphi:G(u)\rightarrow {\mathcal X}(k)$,
$g\mapsto g^{-1}ug$ is surjective, and
the quotient map $\varphi':Z(u)\backslash G(u)\rightarrow {\mathcal X}(k)$ is an isomorphism
of algebraic varieties.
The map $\psi:G(u)\rightarrow {\mathcal F}_u$,
$g\mapsto gF_0$ is well defined
and surjective, and
the quotient map $\psi':G(u)\slash B\rightarrow {\mathcal F}_u$ is an isomorphism
of algebraic varieties (see \cite{M-Pagnon}).

For $\sigma\in \mathbf{S}_n^2(k)$, it is easy to see that $\psi(\varphi^{-1}({\mathcal B}_\sigma))\subset {\mathcal F}_u$
is a $Z(u)$-orbit.
Write $\sigma=(i_1,j_1)\cdots(i_k,j_k)$ with $i_l<j_l$, and let $i_{k+1},...,i_{n-k}$ be the fixed points
of $\sigma$. Let $g_\sigma\in GL(V)$ be defined by
$g_\sigma(e_{i_l})=e_{l}$ for $l=1,...,n-k$,
and $g_\sigma(e_{j_l})=e_{n-k+l}$ for $l=1,...,k$.
It is easy to check that  $g_\sigma\in G(u)$,
and that it satisfies
$N_\sigma=\varphi(g_\sigma)$ and $\psi(g_\sigma)\in{\mathcal Z}_\sigma$.
Thus $\psi(\varphi^{-1}({\mathcal B}_\sigma))={\mathcal Z}_\sigma$.
Then, Claim (a) of Proposition \ref{proposition-orbits} follows from
fact (a) of part (1) of the proof.
Since the maps $\varphi$ and $\psi$ are open, it is easy to see that
$${\mathcal B}_{\sigma'}\subset\overline{{\mathcal B}_\sigma}\quad{\rm if\ and\ only\ if}\quad
{\mathcal Z}_{\sigma'}\subset\overline{{\mathcal Z}_\sigma}.$$
Then, Claim (c) of Proposition \ref{proposition-orbits} follows from fact (c) of part (1) of the proof.
Moreover, since $\varphi$ and $\psi$ are fibrations of fibers $Z(u)$ and $B$ respectively,
we have
\renewcommand{\arraystretch}{1.5}
\[\begin{array}{c}
\mathrm{dim}\,{\mathcal F}_u+\mathrm{dim}\,B=\mathrm{dim}\,{\mathcal X}(k)+\mathrm{dim}\,Z(u),\\
\mathrm{dim}\,{\mathcal B}_\sigma+\mathrm{dim}\,Z(u)=\mathrm{dim}\,\varphi^{-1}({\mathcal B}_\sigma)
=\mathrm{dim}\,\psi^{-1}({\mathcal Z}_\sigma)
=\mathrm{dim}\,{\mathcal Z}_\sigma+\mathrm{dim}\,B.
\end{array}
\renewcommand{\arraystretch}{1}\]
Therefore,  $\mathrm{codim}_{{\mathcal F}_u}\,{\mathcal Z}_\sigma=
\mathrm{codim}_{{\mathcal X}(k)}\,{\mathcal B}_\sigma$,
and Claim (b) of Proposition \ref{proposition-orbits} follows from
the fact (b) of part (1) of the proof.

To complete the proof of Proposition \ref{proposition-orbits},
it remains to show (d).
On one hand, by definition of $\sigma_T$
we have $c(\sigma_T)=b(\sigma_T)=0$, hence,
by (b)
we obtain $\mathrm{dim}\,{\mathcal Z}_{\sigma_T}=\mathrm{dim}\,{\mathcal F}_u$.
Thus, the closure of ${\mathcal Z}_{\sigma_T}$ is a component of ${\mathcal F}_u$.
On the other hand,
by definition of ${\mathcal F}_u^T$ in \ref{parameterization-components}, it is easy
to see that ${\mathcal Z}_{\sigma_T}\subset {\mathcal F}_u^T$,
which implies ${\mathcal Z}_{\sigma_T}\subset{\mathcal K}^T$, and this completes the proof of (d).
\hfill $\Box$

\subsection{The minimal $Z(u)$-orbit}
\label{minimal-orbit}
Let $\sigma_0=\sigma_0(k)\in \mathbf{S}^2_n(k)$ be the permutation
\[\sigma_0=(1,n-k+1)(2,n-k+2)\cdots(k,n).\]
By \cite[Proposition 3.14]{M-P2}, $\B_{\sigma_0}\subset\ov{\B_\sigma}$ for any $\sigma\in\mathbf{S}^2_n(k).$
Thus, by the bijection above, ${\mathcal Z}_{\sigma_0}$ lies in the closure of
each $Z(u)$-orbit of ${\mathcal F}_u$ and it is the unique minimal $Z(u)$-orbit
of ${\mathcal F}_u$.
We have $c(\sigma_0)=\frac{1}{2}k(k-1)$ and $b(\sigma_0)=k(n-2k)$.
By Proposition \ref{proposition-orbits}(b) and formula (\ref{eq1}),
it follows (after simplification) 
$$\mathrm{dim}\,{\mathcal Z}_{\sigma_0}=\frac{1}{2}(n-2k)(n-2k-1)+\frac{1}{2}k(k-1).$$

\section{Decomposition of $Z(u)$-orbits into cells}

\label{section-cell-decomposition}

The main goal of this section is to show that a component
${\mathcal K}^T\subset {\mathcal F}_u$
has a symmetric distribution of Betti numbers
if and only if the
distribution of dimensions of its $Z(u)$-orbits has also
a symmetry property, namely: for all $m$, the number of orbits ${\mathcal Z}_\sigma\subset{\mathcal K}^T$ with
$\mathrm{dim}\,{\mathcal Z}_\sigma=\mathrm{dim}\,{\mathcal K}^T-m$ is equal to
the number of orbits ${\mathcal Z}_{\sigma'}\subset{\mathcal K}^T$
with $\mathrm{dim}\,{\mathcal Z}_{\sigma'}=\mathrm{dim}\,{\mathcal Z}_{\sigma_0}+m$, with $\sigma_0$ from \ref{minimal-orbit}.
To do this, we show that each $Z(u)$-orbit admits a cell decomposition
such that the number of cells and their codimensions are the same for all the orbits.

\subsection{Cell decomposition of an algebraic variety}
\label{cell-decomposition}
Let us begin with general definitions.
A finite partition of an algebraic variety $X$ is said to be an $\alpha$-partition
if the subsets in the partition can be indexed by $1,...,r$ in such a way that
$X_1\cup\ldots\cup X_p$ is closed for every $p=1,...,r$. An $\alpha$-partition is a
{\em cell decomposition} if each subset $X_p$ is isomorphic as an algebraic variety
to an affine space $\mathbb{C}^{d_p}$ for some $d_p\geq 0$.
If $X$ is a projective variety with a cell decomposition $X=X_1\cup\ldots\cup X_r$ as above,
then
 $$\dim H^{m}(X,\mathbb{Q})=\left\{\begin{array}{ll}
 0& {\rm if\ }m=2l+1,\\
|\{p:d_p=l\}|&{\rm if\ }m=2l.\\
\end{array}\right.$$

\subsection{A cell decomposition of the flag variety}
\label{lem-cell-decomposition}
For $w\in\mathbf{S}_n$, let $n_\mathrm{inv}(w)$ denote the inversion number of $w$ (equivalently, its length as an element of the Weyl group).
We need the following

\begin{lem}\label{lemma-cell-decomposition}
Let $V$ be an $n$-dimensional $\mathbb{C}$-vector space and let $W\subset V$ be an $m$-dimensional subspace.
Consider the variety of complete flags ${\mathcal F}={\mathcal F}(V)$.
Let $H=\{g\in GL(V):g(W)=W\}$.
Then, each $H$-orbit ${\mathcal O}\subset {\mathcal F}$
has a cell decomposition ${\mathcal O}=\bigcup C(w,w')$
parameterized by the pairs of permutations $(w,w')\in \mathbf{S}_m\times \mathbf{S}_{n-m}$,
with $\mathrm{dim}\,{\mathcal O}-\mathrm{dim}\,C(w,w')=n_\mathrm{inv}(w)+n_\mathrm{inv}(w')$.
\end{lem}
\begin{proof}
The group $H$ is a parabolic subgroup of $GL(V)$.
To each $H$-orbit ${\mathcal O}$, it corresponds a subset $I_{\mathcal O}\subset\{1,...,n\}$ with $|I_{\mathcal O}|=m$,
such that,
setting $c_i=|\{1,...,i\}\cap I_{\mathcal O}|$ for $i=0,...,n$,
the orbit ${\mathcal O}$ is the set of flags $(V_0,...,V_n)$ with
\[\mathrm{dim}\,V_i\cap W=c_i\quad \forall i=0,...,n.\]
A Borel subgroup $B\subset H$ fixes some complete flag $(W_1\subset ...\subset W_n)$
such that $W_m=W$.
The $B$-orbits of ${\mathcal F}$ form a Schubert cell decomposition
${\mathcal F}=\bigcup_{w\in \mathbf{S}_n}C(w)$,
with $\mathrm{dim}\,C(w)=\mathrm{dim}\,{\mathcal F}-n_\mathrm{inv}(w)$,
and the cell $C(w)$ is the set of flags $(V_0,...,V_n)$ such that
\[
\mathrm{dim}\,V_i\cap W_j=|\{w_1,...,w_i\}\cap \{1,...,j\}|
\quad\forall i,j=0,...,n.\]
Let $\mathbf{S}_n({\mathcal O})=\{w\in \mathbf{S}_n:w_i\in\{1,...,m\}\ \forall i\in I_{\mathcal O}\}$.
The cell $C(w)$ is contained in ${\mathcal O}$ if
and only if $w \in \mathbf{S}_n({\mathcal O})$.
Write $I_{\mathcal O}=\{i_1<...<i_m\}$ and $\{1,...,n\}\setminus I_{\mathcal O}=\{i_{m+1}<...<i_n\}$.
Let $w_{\mathcal O}\in \mathbf{S}_n({\mathcal O})$ be defined by $w_{\mathcal O}:i_l\mapsto l$.
Let $\mathbf{S}_n[m]$ denote the set of permutations $w\in \mathbf{S}_n$ such that $\{w_1,...,w_m\}=\{1,...,m\}$.
The map $\varphi:\mathbf{S}_m\times \mathbf{S}_{n-m}\rightarrow \mathbf{S}_n[m]$, $(w,w')\mapsto \varphi(w,w')$
defined by $\varphi(w,w'):i\mapsto w_i$
for $i=1,...,m$ and $\varphi(w,w'):m+i\mapsto m+w'_i$ for $i=1,...,n-m$, is a bijection
which satisfies $n_\mathrm{inv}(\varphi(w,w'))=n_\mathrm{inv}(w)+n_\mathrm{inv}(w')$.
The map $\mathbf{S}_n[m]\rightarrow \mathbf{S}_n({\mathcal O})$, $w\mapsto ww_{\mathcal O}$
is a bijection, moreover we have
$n_{\mathrm{inv}}(ww_{\mathcal O})=n_{\mathrm{inv}}(w)+n_{\mathrm{inv}}(w_{\mathcal O})$.
It follows that $C(w_{\mathcal O})$ is the open cell of ${\mathcal O}$.
The sets $C(w,w'):=C(\varphi(w,w')w_{\mathcal O})$
for $(w,w')\in \mathbf{S}_m\times \mathbf{S}_{n-m}$
form a cell decomposition of ${\mathcal O}$ and we have
$\mathrm{dim}\,{\mathcal O}-\mathrm{dim}\,C(w,w')=n_\mathrm{inv}(w)+n_\mathrm{inv}(w')$.
\end{proof}

\subsection{Cell decomposition of an orbit ${\mathcal Z}_\sigma$}
\label{proposition-cell-decomposition}
Recall that the diagram $Y(u)$ representing the Jordan form of $u$ has two
columns of lengths $(n-k,k)$. It follows
$\mathrm{dim}\,\mathrm{ker}\,u=n-k$ and $\mathrm{dim}\,\mathrm{Im}\,u=k$.
Let ${\mathcal F}(\mathrm{ker}\,u)$
be the variety of complete flags $(V_0\subset...\subset V_{n-k}=\mathrm{ker}\,u)$.
More generally, a sequence of subspaces
$(V_0\subseteq...\subseteq V_n= \mathrm{ker}\,u)$
which contains as a subsequence a complete flag
$(V_0,V_{i_1},...,V_{i_{n-k}})\in{\mathcal F}(\mathrm{ker}\,u)$
is considered as an element of ${\mathcal F}(\mathrm{ker}\,u)$.
Let $H\subset GL(\mathrm{ker}\,u)$ be the parabolic subgroup
of elements $g$ such that
$g(\mathrm{Im}\,u)=\mathrm{Im}\,u$.

\begin{prop}\label{prop-cell-decomposition}
Let ${\mathcal Z}\subset{\mathcal F}_u$ be a $Z(u)$-orbit. The map
\[\varphi:{\mathcal Z}\rightarrow {\mathcal F}(\mathrm{ker}\,u),\ (V_0,...,V_n)\mapsto (V_0\cap \mathrm{ker}\,u,...,V_n\cap \mathrm{ker}\,u)\]
is well defined and algebraic.
Its image $\varphi({\mathcal Z})$ is a $H$-orbit of ${\mathcal F}(\mathrm{ker}\,u)$,
and the map $\varphi:{\mathcal Z}\rightarrow \varphi({\mathcal Z})$ is
a vector bundle.
In particular, ${\mathcal Z}$
has a cell decomposition
\[{\mathcal Z}=\bigcup_{\mathbf{S}_k\times \mathbf{S}_{n-2k}} C_{\mathcal Z}(w,w')\]
parameterized by the pairs of permutations $(w,w')\in \mathbf{S}_k\times \mathbf{S}_{n-2k}$,
with $\mathrm{dim}\,C_{\mathcal Z}(w,w')=\mathrm{dim}\,{\mathcal Z}-n_{\mathrm{inv}}(w)-n_{\mathrm{inv}}(w')$.
In particular, the number of cells and their codimensions do not depend on the orbit ${\mathcal Z}$.
\end{prop}
\begin{proof}
Let $\sigma\in \mathbf{S}_n^2(k)$ and let ${\mathcal Z}={\mathcal Z}_\sigma$ be the corresponding $Z(u)$-orbit.
Let $P_\sigma$ be the link pattern of $\sigma$.
Let $P_\sigma^-$, $P_\sigma^+$, $P_\sigma^0$ be respectively the
sets of left end points, right end points and fixed points (see \ref{definition-link-pattern}).
Write
$P_\sigma^-\cup P_\sigma^0=\{i_1<...<i_{n-k}\}$.
Then $(V_0\cap \mathrm{ker}\,u,V_{i_1}\cap \mathrm{ker}\,u,...,
V_{i_{n-k}}\cap\mathrm{ker}\,u)$ is a complete flag of $\ker u$.
The map $\varphi$ is thus well defined and algebraic.

For $l=1,...,n-k$, let
$c_l^-=|\{i_1,...,i_l\}\cap P_\sigma^-|$.
Let ${\mathcal O}={\mathcal O}(\sigma)$ be the set of flags $(V_0,...,V_{n-k})\in{\mathcal F}(\mathrm{ker}\,u)$
with
\[\mathrm{dim}\,V_l\cap \mathrm{Im}\,u=c_l^-\quad \forall l=0,...,n-k.\]
The set ${\mathcal O}$ is a $H$-orbit of ${\mathcal F}(\mathrm{ker}\,u)$
and we see that ${\mathcal O}$ is the image of $\varphi$.

Choose $\omega:\mathrm{Im}\,u\rightarrow V$
such that $u\circ\omega(x)=x$ for any $x\in \mathrm{Im}\,u$.
Thus we have $V=\mathrm{ker}\,u\oplus \mathrm{Im}\,\omega$.
For $t\in\mathbb{C}^*$, define $h_t\in GL(V)$
by $h_t(x+y)=x+ty$ for $x\in\mathrm{ker}\,u$ and
$y\in\mathrm{Im}\,\omega$.
We have $h_tuh_{t}^{-1}=t^{-1}u$, hence the group
$(h_t)_{t\in\mathbb{C}^*}$ linearly acts on ${\mathcal F}_u$.
Let ${\mathcal Z}^{(h)}$ be the set of flags $F=(V_0,...,V_n)\in {\mathcal Z}$
which are fixed by $(h_t)_{t\in\mathbb{C}^*}$.
Equivalently ${\mathcal Z}^{(h)}$ is the set of flags
$F=(\langle e_1,...,e_i\rangle)_{i=0,...,n}$ for some basis
$(e_1,...,e_n)$ such that $e_i\in \mathrm{ker}\,u$
or $e_i\in \mathrm{Im}\,\omega$ depending on whether
$i\in P_\sigma^0\cup P_\sigma^-$ or $i\in P_\sigma^+$.
We have ${\mathcal Z}=\{F\in {\mathcal F}_u:\mathrm{lim}_{t\rightarrow \infty}h_t\cdot F\in{\mathcal Z}^{(h)}\}$.
Note that the action of $(h_t)_{t\in\mathbb{C}^*}$ leaves ${\mathcal Z}$ invariant.
Moreover, ${\mathcal Z}$ is smooth and irreducible.
By \cite[Theorem 4.1]{Bialynicki-Birula},
the map $\pi:{\mathcal Z}\rightarrow {\mathcal Z}^{(h)}$, $F\mapsto(\mathrm{lim}_{t\rightarrow\infty}h_t\cdot F)$
is an algebraic vector bundle.

Next, we prove that the restriction $\varphi^{(h)}:{\mathcal Z}^{(h)}\rightarrow {\mathcal O}$
is also a vector bundle.
It will follow that $\varphi=\varphi^{(h)}\circ\pi$ is a vector bundle.

Fix $F_0\in {\mathcal O}$, and write
$F_0=(\langle e_1,...,e_i\rangle)_{i=0,...,n-k}$ where
$(e_1,...,e_{n-k})$ is some basis of $\mathrm{ker}\,u$
with $e_j\in \mathrm{Im}\,u$ if $i_j\in P_\sigma^-$.
Let $U=\{g\in H:ge_i-e_i\in\langle e_{i+1},...,e_{n-k}\rangle\}$.
The map $\xi:U\rightarrow {\mathcal O}$, $g\mapsto gF_0$ is an open
immersion.
Let us show that $\varphi^{(h)}$ is trivial over $\xi(U)$.

For $j\in P_\sigma^+$, let $I_j=\{j'\in P_\sigma^+:j'>j,\ \sigma(j')<\sigma(j)\}$.
Let $I=\{(j,j'): j\in P_\sigma^+\mbox{ and }j'\in I_j\}$.
For $g\in U$ and $\underline{t}=(t_{j,j'})\in \mathbb{C}^I$, we define
vectors $f_i=f_i(g,\underline{t})$
forming a basis $(f_i)_{i=1,...,n}$.
For $i_j\in P_\sigma^-\cup P_\sigma^0$ write $f_{i_j}=ge_j$.
For $j\in P_\sigma^+$ write
$$\begin{array}{c} f_j=\omega(\,f_{\sigma(j)}+\sum\limits_{j'\in I_j}t_{j,j'}f_{\sigma(j')}\,). \end{array}$$
The map
$$\Phi:U\times \mathbb{C}^{I}\rightarrow {\varphi^{(h)}}^{-1}(\xi(U)),\
(g,\underline{t})\mapsto (\langle f_1,...,f_i\rangle)_{i=0,...,n}$$
is well defined, we have $\varphi^{(h)}\circ\Phi=\xi\circ \mathrm{pr}_1$,
and $\Phi$ is an isomorphism of algebraic varieties.

Thus, $\varphi^{(h)}$ is a locally trivial fibration of fiber $\mathbb{C}^{I}$.
It is easy to see that the chart changes are linear.
It follows that $\varphi^{(h)}$ is a vector bundle.

The last claim of the statement then follows from Lemma \ref{lemma-cell-decomposition}.
\end{proof}

\subsection{Poincar\'e polynomial of a component ${\mathcal K}^T$}
\label{proposition-Betti-orbits}
Note that the $Z(u)$-orbits of the Springer fiber ${\mathcal F}_u$ (resp. of the component ${\mathcal K}^T\subset{\mathcal F}_u$)
form an $\alpha$-partition. Thus, by collecting together the cell decompositions
of all the $Z(u)$-orbits of ${\mathcal F}_u$ (resp. of ${\mathcal K}^T$),
we get a cell decomposition of ${\mathcal F}_u$ (resp. of ${\mathcal K}^T$).
Let $d=\mathrm{dim}\,{\mathcal K}^T$.
Let $b_m^T=\mathrm{dim}\,H^{2d-m}({\mathcal K}^T,\mathbb{Q})$.
We deduce a formula for the Poincar\'e polynomial
$P^T(x):=\sum_m b_{m}^Tx^m$.

Let ${\mathcal Z}_{\sigma_0}$ be the minimal $Z(u)$-orbit
and let $d_0=\mathrm{dim}\,{\mathcal Z}_{\sigma_0}$ be its dimension
(see \ref{minimal-orbit}).

Let $n(T:m)$ be the number of $m$-codimensional $Z(u)$-orbits of ${\mathcal K}^T$.
We have thus $n(T:m)=0$ for $m\notin\{0,...,d-d_0\}$.
Write
$$N^T(x)=\sum\limits_{m=0}^{d-d_0}n(T:m)\,x^m.$$

Put $[p]_x=1+x+...+x^{p-1}$ and $[p]_x!=[1]_x\cdots [p]_x$
(and, by convention, $[0]_x!=1$). Let $j_n(p)=|\{w\in \bS_n \ :\ n_{\mathrm{inv}}(w)=p\}|.$
Let $Q_n(x)=\sum_p j_n(p)x^p$. Then, as one can easily see (or cf. \cite{Sag} for example),
 $Q_n(x)=[n]_x!.$

Let $i(p)=|\{(w,w')\in \mathbf{S}_k\times \mathbf{S}_{n-2k}:n_\mathrm{inv}(w)+n_\mathrm{inv}(w')=p\}|$.
Set
$I(x)=\sum_p i(p)x^p$, then
$$I(x)=Q_k(x)Q_{n-2k}(x)=[k]_x!\,[n-2k]_x!\,.$$

By \ref{cell-decomposition} and
Proposition \ref{prop-cell-decomposition}, we have the following

\begin{cor}
$P^T(x)=N^T(x^2)I(x^2)$.
\end{cor}

We derive:

\begin{prop}
\label{prop-Betti-orbits}
The component ${\mathcal K}^T$ has a symmetric distribution of Betti numbers
if and only if
 $n(T:m)=n(T:d-d_0-m)$ for every $m=0,...,d-d_0$.
\end{prop}
\begin{proof}
We have $b_m^T=b_{2d-m}^T$, $\forall m=0,...,2d$,
if and only if $P^T(x)=x^{2d}P^T(x^{-1})$.
Equivalently, we have $N^T(x^2)I(x^2)=x^{2d}N^T(x^{-2})I(x^{-2})$.
Observe that $I(x^2)=x^{2d_0}I(x^{-2})$, hence
the property is equivalent to $N^T(x^2)=x^{2d-2d_0}N^T(x^{-2})$.
In other words, it is equivalent to $n(T:m)=n(T:d-d_0-m)$, $\forall m=0,...,d-d_0$.
\end{proof}

\section{On smooth components of  ${\mathcal F}_u$}

\label{section-on-nonsingular-components}

The purpose of this section is to prove the part of Theorem \ref{first-crit}
corresponding to the smooth case.

\subsection{A few preliminary notes}
To begin with, we make some observations about Theorem \ref{previous-crit}.

\subsubsection{Cardinality of the set $X(\tau_0)$}
\label{observation-1}
Let $\tau$ be a row-standard tableau and let $F_\tau$ be its flag (cf. \ref{previous-criterion}). Let $(i_1,...,i_{n-k})$
and $(j_1,...,j_k)$
be the entries
from top to bottom
in the two columns of $\tau$.
Set $\sigma(\tau)=(i_1,j_1)\cdots(i_k,j_k)\in \mathbf{S}_n^2(k)$.
Then, obviously, $F_\tau\in{\mathcal Z}_{\sigma(\tau)}$.
By Proposition \ref{proposition-orbits}.(c)-(d),  $F_\tau\in{\mathcal K}^T$ if and only if
$\sigma(\tau)\preceq \sigma_T$.

Now, let us consider the set $X(\tau_0)$ involved in \ref{previous-criterion}. We partition it into five parts.
\begin{enumerate}
\item There are $\frac{1}{2}k(k-1)$ tableaux $\tau\in X(\tau_0)$
obtained from $\tau_0$ by switching
$i,j$ with $1\leq i<j\leq k$.
\item There are $k(n-2k)$ tableaux $\tau\in X(\tau_0)$
obtained from $\tau_0$ by switching
$i,j$ with $1\leq i\leq k<j\leq n-k$. In particular, for any $\tau$ in that case,
there exists $j'\ :\ j'\geq n-k+1$ such that $(j,j')\in\sigma(\tau).$
\item There are $\frac{1}{2}k(k-1)$ tableaux $\tau\in X(\tau_0)$
obtained from $\tau_0$ by switching
$i,j$ with $1\leq i\leq k$ and $n-k+1\leq j< n-k+i$. In particular, for any $\tau$ in that case,
there exists $i'\ :\ 1\leq i'<i\leq k$ such that $(i',i),(n-k+i',n-k+i)\in\sigma(\tau).$
\item There are $k(n-2k)$ tableaux $\tau\in X(\tau_0)$
obtained from $\tau_0$ by switching
$i,j$ with $k+1\leq i\leq n-k$ and $n-k+1\leq j\leq n$.  In particular, for any $\tau$ in that case,
there exists $i'\ :\ 1\leq i'\leq k$ such that $(i',i)\in\sigma(\tau).$
\item There are $\frac{1}{2}(n-2k)(n-2k-1)$ tableaux $\tau\in X(\tau_0)$
obtained from $\tau_0$ by switching $i,j$ with $k+1\leq i<j\leq n-k$.
For such $\tau$, observe that $\sigma(\tau)=\sigma_0$ and $F_\tau\in{\mathcal Z}_{\sigma_0}$,
hence $F_\tau\in{\mathcal K}^T$ for every $T$ standard.
\end{enumerate}
Finally, we get
\begin{eqnarray}\label{eq_tau_0}
|X(\tau_0)|=k(n-2k)+\frac{1}{2}k(k-1)+\frac{1}{2}(n-k)(n-k-1)
\end{eqnarray}

\subsubsection{Inductive property}
\label{observation-2}
Let $T$ be a standard tableau with two columns of lengths $(n-k,k)$
such that $n$ belongs to the second column, that is  $c_T(n)=2$.
Let $T'=T_{n-1}$. Let ${\mathcal K}^{T'}$ be the component associated
to the tableau $T'$ in $\F_{u'}$ where $Y(u')$ has two columns of lengths $(n-k,k-1).$
We need the following simple
\begin{lem}\label{observ-2} Let $T$ be such that $c_T(n)=2$ and let $T'=T_{n-1}$.
If ${\mathcal K}^{T'}$ is smooth, then
${\mathcal K}^T$ is smooth.
\end{lem}
\begin{proof}
Let $\tau'_0$ be the subtableau of $\tau_0$ of entries $1,...,n-1$, and let $X(\tau'_0)$ be the
set of adjacent row-standard tableaux relative to it, in the sense of \ref{previous-criterion}.
Let $X'(\tau_0)\subset X(\tau_0)$ be the subset of tableaux  with $n$ in the last position of the second
column. For $\tau\in X'(\tau_0)$, let $\tau'$ be the subtableau of entries $1,...,n-1$. The map
$X'(\tau_0)\rightarrow X(\tau'_0)$, $\tau\mapsto \tau'$ is a bijection.
By Proposition \ref{proposition-orbits}.(c)-(d),
 if $\tau\in X'(\tau_0)$ is such that
$F_\tau\in {\mathcal K}^T$, then $F_{\tau'}\in {\mathcal K}^{T'}$.
For $\tau\in X(\tau_0)\setminus X'(\tau_0)$, one has $R_{1,n-1}(\sigma(\tau))=k$
whereas $R_{1,n-1}(\sigma_T)=k-1$, hence $\sigma(\tau)\not\preceq \sigma_T$, thus $F_\tau\notin{\mathcal K}^T$.
We get $|\{\tau\in X(\tau_0):F_\tau\in{\mathcal K}^T\}|\leq |\{\tau\in X(\tau'_0):F_\tau\in{\mathcal K}^{T'}\}|$.
If ${\mathcal K}^{T'}$ is smooth then, by Theorem \ref{previous-crit}, $|\{\tau\in X(\tau'_0):F_\tau\in{\mathcal K}^{T'}\}|\leq{\frac1 2}((n-1)-(k-1))((n-1)-(k-1)-1)={\frac1 2}(n-k)(n-k-1)$ so that $|\{\tau\in X(\tau_0):F_\tau\in{\mathcal K}^T\}|\leq {\frac1 2}(n-k)(n-k-1)$. Hence, by Theorem \ref{previous-crit}, ${\mathcal K}^T$ is smooth.
\end{proof}

\begin{rem}
{\rm The property in the lemma is true more generally. Let $T$ be a standard tableau of general shape.
Let $T'=T_{n-1}$. Let ${\mathcal K}^T$
and ${\mathcal K}^{T'}$ be the components associated to $T$ and $T'$ in the corresponding
Springer fibers. Then, by \cite[Theorem 2.1]{F-M}, if the component ${\mathcal K}^{T'}$ is singular, then ${\mathcal K}^{T}$ is singular.
Moreover, if $n$ lies in the last column of $T$, then ${\mathcal K}^{T}$ is singular if and only if ${\mathcal K}^{T'}$ is singular
(see also Proposition \ref{fiber-bundle}).}
\end{rem}

\subsection{Sch\"utzenberger involution}
\label{observation-3}
Next, we point out a natural duality on components.
Let $\sigma=(i_1,j_1)\cdots(i_k,j_k)\in \mathbf{S}_n^2(k)$.
We define
$$\sigma^*=(n+1-j_1,n+1-i_1)\cdots (n+1-j_k,n+1-i_k).$$
We get an involutive transformation $\mathbf{S}_n^2(k)\rightarrow \mathbf{S}_n^2(k)$,
$\sigma\mapsto \sigma^*$. Note that $P_{\sigma^*}$ is a mirror picture of $P_\sigma$. In particular,
one has $\sigma\preceq \omega$ if and only if $\sigma^*\preceq \omega^*.$

Let $V^*$ be the dual space of $V$. For a subspace $W\subset V$,
let $W^\perp=\{\phi\in V^*:\phi(w)=0,\ \forall w\in W\}\subset V^*$.
Let $u^*:V^*\rightarrow V^*$, $\phi\mapsto \phi\circ u$, this is a nilpotent endomorphism
of same Jordan form as $u$.
Let ${\mathcal F}_{u^*}$ be the Springer fiber associated to $u^*$.
Let $Z(u^*)=\{g\in GL(V^*):gu^*=u^*g\}$.
Let $\{{\mathcal Z}^*_{\sigma}:\sigma\in\mathbf{S}_n^2(k)\}$ be the $Z(u^*)$-orbits of ${\mathcal F}_{u^*}$,
in the sense of \ref{sect-proposition-orbits}.
The map
$$\Phi:{\mathcal F}_u\rightarrow{\mathcal F}_{u^*},\ (V_0,...,V_n)\mapsto (V_n^\perp,...,V_0^\perp)$$
is an isomorphism of algebraic varieties.
It is easy to see that $\Phi$ maps ${\mathcal Z}_\sigma$ onto ${\mathcal Z}^*_{\sigma^*}$, for every $\sigma\in\mathbf{S}_n^2(k)$.

For $T$ standard with second column $(j_1,\ldots j_k)$ and with $\sigma_T=(i_1,j_1)\cdots(i_k,j_k)$, we have $(\sigma_T)^*=\sigma_{T^S}$ for some other standard tableau $T^S$,
where the second column of $T^S$ is $\{n+1-i_1,\ldots,n+1-i_k\}$ (after ordering in increasing order).
Thus, $\Phi$ induces an isomorphism between the components $\overline{{\mathcal Z}_{\sigma_T}}\subset {\mathcal F}_{u}$ and
$\overline{{\mathcal Z}^*_{\sigma_{T^S}}}\subset {\mathcal F}_{u^*}$.
We have $\sigma\preceq \sigma_T$ if and only if $\sigma^*\preceq \sigma_{T^S}$, hence
there is a one-to-one correspondence between orbits of both components ${\mathcal K}^T$ and ${\mathcal K}^{T^S}$,
which preserves the codimension.

Note that $\tau^*(T)=\{i\ :\ (i,i+1)\in \sigma_T\}$ so that
$\tau^*(T^s)=\{n-i\ :\ i\in\tau^*(T)\}$,
so that $|\tau^*(T)|=|\tau^*(T^S)|$.
Note also that
$1\in P_{\sigma_T}^-$ if and only if $n\in P_{\sigma_{T^S}}^+$,
and $n\in P_{\sigma_T}^+$ if and only if $1\in P_{\sigma_{T^S}}^-$.

\begin{rem}
{\rm The transformation $T\mapsto T^S$ is in fact the classical Sch\"utzenberger involution.
Notice that, if $T$ is a standard tableau of general shape,
denoting by $T^S$ its Sch\"utzenberger transform, then it also holds that the components
${\mathcal K}^T$ and ${\mathcal K}^{T^S}$ are isomorphic (see \cite{vanLeeuwen}).
Nevertheless, the alternative construction of the Sch\"utzenberger involution presented in this section is specific to
the two-column case.}
\end{rem}

\subsection{Smooth components of $\F_u$}
\label{proposition-nonsingular-components}
Now, let us show

\begin{prop}\label{prop-nonsingular-components}
Let $T$ be standard. If $T$ satisfies one of the following conditions:
\begin{itemize}
\item[\rm(i)]  $|\tau^*(T)|=1;$
\item[\rm(ii)]  $|\tau^*(T)|=2$ and at least one of $1,n$ is an end point of $\sigma_T;$
\item[\rm(iii)]  $|\tau^*(T)|=3$ and $1,n$ are end points of $\sigma_T$ but $(1,n)\not\in\sigma_T;$
\end{itemize}
then ${\mathcal K}^T$ is smooth.
\end{prop}
\begin{proof}
{(i)} Note that, if $|\tau^*(T)|=1$, then $\tau^*(T)=\{l\}$ where $k\leq l\leq n-k$.
The entries of the second column of $T$ are $l+1,...,l+k$.
It follows
\[\sigma_T=(l,l+1)(l-1,l+2)\cdots (l-k+1,l+k).\]

In particular, $R_{1,l}(\sigma_T)=R_{l+1,n}(\sigma_T)=0$.

Let us consider  $\tau\in X(\tau_0)$ described in \ref{observation-1} (2),(3),(4). Note that
\begin{itemize}
\item{} If $\tau\in X(\tau_0)$ is obtained as described in \ref{observation-1}(3), then there exist
$i,i'\ :\ 1\leq i'<i\leq k$ such that $(i',i)\in\sigma(\tau)$. Thus,
 $R_{1,l}(\sigma(\tau))\geq 1$ and
 $\sigma(\tau)\not\preceq \sigma_T$. We get $F_\tau\not\in{\mathcal K}^T$ for any $\tau$ of type \ref{observation-1}(3).
\item{} For $\tau\in X(\tau_0)$ obtained as described in \ref{observation-1}(4)
with $i: k+1\leq i \leq l$,
one has $(i',i)\in\sigma(\tau)$ for some $i'\leq k$.
Thus, $R_{1,l}(\sigma(\tau))\geq 1$ and
 $\sigma(\tau)\not\preceq \sigma_T$. One has $F_\tau\not\in{\mathcal K}^T$.
There are $k(l-k)$ such tableaux $\tau$.
\item{} For $\tau\in X(\tau_0)$ obtained as described in \ref{observation-1}.(2)
with $j>l$,
we have $(j,j')\in\sigma(\tau)$ for some $j'\geq n-k+1$.
Thus, $R_{l+1,n}(\sigma(\tau))\geq 1$ and
 $\sigma(\tau)\not\preceq \sigma_T$. One has $F_\tau\not\in{\mathcal K}^T$.
There are $k(n-k-l)$ such tableaux $\tau$.
\end{itemize}
All together we get that $|\{\tau\in X(\tau_0):F_\tau\not\in{\mathcal K}^T\}|\geq \frac{1}{2}k(k-1)+k(n-2k)$.
Thus, by equality (\ref{eq_tau_0}), $|\{\tau\in X(\tau_0):F_\tau\in{\mathcal K}^T\}|\leq \frac{1}{2}(n-k)(n-k-1)$.
Hence,
by Theorem \ref{previous-crit}, the component ${\mathcal K}^T$ is smooth.

\smallskip
\noindent
{(ii)}
Let $|\tau^*(T)|=2.$ Let us show that, if either $1$ or $n$ is an end point
of $\sigma_T$, then ${\mathcal K}^T$ is smooth.
By Section \ref{observation-3}, we can suppose that $n$ is an end point, hence $n$ lies
in the second column of $T$.
For $n=4$ the statement is true by Theorem \ref{previous-crit}. So assume it is true for
$n-1$ and show for $n$.
Let $T'=T_{n-1}$.
By Lemma \ref{observ-2}, it is sufficient to prove that the component ${\mathcal K}^{T'}$
is smooth.
Only two situations are possible:
\begin{itemize}
\item{} Either $n-1$ is in the first column of $T$,
so that $n-1\in\tau^*(T)$ and in this case $|\tau^*(T')|=1$.
Thus, by (i), ${\mathcal K}^{T'}$ is smooth.
\item{} Or $n-1$ is in the second column of $T'$, and in this case $|\tau^*(T')|=2$ and $n-1$
is an end point of $\sigma_{T'}$. By induction hypothesis,  ${\mathcal K}^{T'}$ is smooth.
\end{itemize}

\smallskip
\noindent
{(iii)}
We show (iii) by induction. For $n=6$ there is a unique
\[T=\young(12,34,56)\]
satisfying: $|\tau^*(T)|=3$. Note that $1,6$ are end points of ${\sigma_T}$ but $(1,6)\not\in\sigma_T$ in this case.
Note that ${\mathcal K}^{T_5}$ is smooth by (ii), thus by Lemma \ref{observ-2}, ${\mathcal K}^T$ is smooth.
Assume (iii) is true for $n-1$ and show for $n$.
Let $(1,i),(j,n)\in\sigma_T$.
Let $T'=T_{n-1}$.
Since $n$ is an end point of $\sigma_T$, it lies in the second column of $T$.
By Lemma \ref{observ-2}, it is enough to show that ${\mathcal K}^{T'}$ is smooth.
Consider $T'$. We still have $(1,i)\in \sigma_{T'}$, hence $1$ is an end point of $\sigma_{T'}$.
\begin{itemize}
\item{} If $j=n-1$, then $|\tau^*(T')|=2$
and $1$ is an end point, so that by (ii) ${\mathcal K}^{T'}$ is smooth.
\item{} If $j<n-1$, then $|\tau^*(T')|=3$.
In this case, $n-1$ is in the second column of $T$.
It follows from the definition of $\sigma_T$ that
$(j',n-1)\in\sigma_{T}$ where $j'>j>1$. Hence $(j',n-1)\in\sigma_{T'}$.
Thus $T'$ satisfies the conditions of induction hypothesis. Thus, ${\mathcal K}^{T'}$ is smooth.
\end{itemize}
The proof is now complete.
\end{proof}

\section{On singular components of  ${\mathcal F}_u$}

\label{section-on-singular-components}

In this section, we consider a standard tableau $T$ of singular type
according to the description of Theorem \ref{first-crit}.
For such a $T$, we show that the number of $Z(u)$-orbits of
codimension $1$ lying in the component ${\mathcal K}^T$ is bigger
than the number of $Z(u)$-orbits of dimension $d_0+1$ in the whole
Springer fiber ${\mathcal F}_u$.
According to Proposition \ref{prop-Betti-orbits}, we deduce that the distribution of
Betti numbers of ${\mathcal K}^T$ is not symmetric.
It follows that ${\mathcal K}^T$ is singular.
It completes the proofs of Theorems \ref{first-crit} and \ref{second-crit}.
We also deduce the proof of Theorem \ref{third-crit}.

\subsection{On the combinatorics of link patterns}
\label{notation-link-patterns}
First, we recall from \cite{Ml-p} some terminology connected to the combinatorics of link patterns,
that we need in the proof.

\subsubsection{Permutations $\sigma_{i\rightarrow p}$ and $\sigma_{i\leftrightarrows j}$ adjacent to $\sigma$}

Let $\sigma\in\mathbf{S}_n^2(k)$ and let $P_\sigma$ be its link pattern,
and as in \ref{definition-link-pattern}, let $P_\sigma^-$ (resp. $P_\sigma^+$, and $P_\sigma^0$) be the set
of left end points (resp. right end points, and fixed points) of $P_\sigma$.

For $i\in P_\sigma^-\cup P_\sigma^+$ there is $j\not=i$ such that $(i,j)\in\sigma$ or $(j,i)\in\sigma$. In what follows we put an arc in double brackets $((i,j))$ if we do not know the ordering of $i,j.$
Let $p\in P_\sigma^0$. We set $\sigma_{i\rightarrow p}$ to be
the involution obtained from $\sigma$ by changing $((i,j))$ to $((p,j))$, i.e.,
\[\mbox{if }\sigma=(i_1,j_1)\cdots(i_{k{-}1},j_{k{-}1})((i,j)),\mbox{ then }\sigma_{i\rightarrow p}=(i_1,j_1)\cdots(i_{k{-}1},j_{k{-}1})((p,j)).\]
Note that we cannot say anything about ordering
inside $((i,j))$ and $((p,j))$.

Let $i,j\in P_\sigma^-\cup P_\sigma^+$. Assume that they belong to different pairs: $((i,p)),((j,q))\in\sigma$ and $((i,j))\not\in\sigma $.
Put $\sigma_{i\leftrightarrows j}$ to be the involution obtained from $\sigma$
by interchanging the places of $i$ and $j$ in the pairs, that is:
\[\mbox{if }\sigma=(i_1,j_1)\cdots(i_{k{-}2},j_{k{-}2})((i,p))((j,q)),\mbox{ then }\sigma_{i\leftrightarrows j}=(i_1,j_1)\cdots(i_{k{-}2},j_{k{-}2})((j,p))((i,q)).\]

\subsubsection{Concentric and consecutive pairs of arcs, next point of an arc}\label{new-rem}
Write $\sigma=(i_1,j_1)\cdots (i_k,j_k)$ with $i_s<j_s$ for every $s=1,...,k$.
We say that $(i_s,j_s)$ is over $(i_t,j_t)$, or equivalently that $(i_t,j_t)$ is under $(i_s,j_s)$,
if we have $i_s<i_t<j_t<j_s$. We call an arc $(i,j)$ minimal if there are no arcs under it.
Let $(i_t,j_t)$ be under $(i_s,j_s)$. We say that the pair $\{(i_s,j_s),(i_t,j_t)\}$ is concentric
if every $(i_q,j_q)\ne(i_s,j_s)$ over $(i_t,j_t)$ is also over $(i_s,j_s)$.

Let $(i_s,j_s)$ be on the left of $(i_t,j_t)$, i.e. $(i_s,j_s)\in[1,i_t]$. We call them consecutive
if $\sigma$ has no fixed point in the interval $[j_s,i_t]$ and each $(i_q,j_q)$ over one of them satisfies
$i_q,j_q\notin[j_s,i_t]$.

A fixed point $p\in[1,i_s]$ is called the next point
on the left of $(i_s,j_s)$ if two conditions are satisfied:
$p$ is the only fixed point on the interval $[p,i_s]$,
and any arc $(i_q,j_q)$ over $(i_s,j_s)$ is a bridge over $p.$

Respectively, a fixed point $p\in[j_s,n]$
is called the next point on the right of  $(i_s,j_s)$
if two conditions are satisfied:
$p$ is the only fixed point on the interval $[j_s,p]$, and
any arc $(i_q,j_q)$ over $(i_s,j_s)$ is a bridge over $p$.

\smallskip
Example:
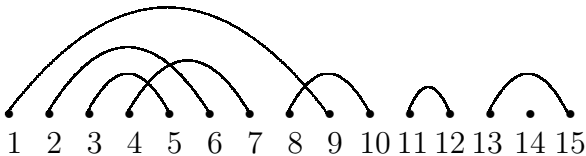
\begin{center}
\begin{picture}(100,90)(0,30)
\multiput(-139,60)(15,0){15}%
{\circle*{3}}
 \put(-140,45){1}
 \put(-125,45){2}
 \put(-110,45){3}
 \put(-95,45){4}
 \put(-80,45){5}
 \put(-65,45){6}
 \put(-50,45){7}
 \put(-35,45){8}
\put(-20,45){9}
 \put(-8,45){10}
 \put(6,45){11}
 \put(20,45){12}
 \put(34,45){13}
 \put(50,45){14}
\put(65,45){15}
 \qbezier(-125,60)(-95,110)(-64,60)
 \qbezier(-110,60)(-95,90)(-79,60)
\qbezier(-95,60)(-72.5,100)(-49,60)
 \qbezier(-35,60)(-20,90)(-4,60)
\qbezier(10,60)(17.5,80)(26,60)
\qbezier(-140,60)(-80,140)(-20,60)
\qbezier(40,60)(55,90)(70,60)
\end{picture}
\end{center}
In our example, $(2,6)$ is  concentric over $(3,5)$;
$(1,9)$ is  concentric over both $(2,6)$ and $(4,7)$ and these are all the
 concentric pairs. Also,  $(1,9)$ is consecutive with both $(11,12)$ and $(13,15)$.
Also both $(2,6)$ and $(4,7)$ are consecutive with $(8,10)$;   and $(8,10)$ is also consecutive with both $(11,12)$  and $(13,15)$;
finally, $(11,12)$ and $(13,15)$ are consecutive and these are all the consecutive pairs. The only fixed point of the pattern is $14$. It is the next point on the right of  $(1,9),\ (8,10),\ (11,12).$
The minimal arcs are $\{ (3,5),\, (4,7),\, (8,10),\, (11,12),\, (13,15)\}.$

\begin{rem}
\label{in512}
{\rm Let $T$ be a standard two-column tableau. By definition, $(i,i+1)\in\sigma_T$ for $i\in\tau^*(T).$ Moreover, since $P_{\sigma_T}$ is a link pattern without fixed points under the arcs and without arc crossings,
the minimal arcs of $P_{\sigma_T}$ are exactly $\{(i,i+1)\ :\ i\in\tau^*(T)\}.$}
\end{rem}

\subsubsection{Description of 1-codimensional inclusions of orbit closures}
\label{lemma-orbits-codimension-1}
Combining Propositions 3.4, 3.5, 3.6, Theorem 3.7 of \cite{Ml-p} and part (2)
of the proof of Proposition \ref{proposition-orbits}, we get:

\begin{lem}\label{lem-codim} Let $\sigma\in\bS_n^2(k)$.
 Let $\sigma'$ be obtained from $\sigma$ in one of the following ways:
\begin{itemize}
\item[\rm (i)] For $p$  the next point on the left of  $(i_s,j_s)\in\sigma$, put $\sigma'=\sigma_{i_s\rar p};$
\item[\rm (ii)] For $p$  the next point on the right of  $(i_s,j_s)\in\sigma$, put $\sigma'=\sigma_{j_s\rar p};$
\item[\rm (iii)] For $(i_s,j_s),(i_t,j_t)$  consecutive with $j_s<i_t$, put $\sigma'=\sigma_{j_s\leftrightarrows i_t};$
\item[\rm (iv)]  For $(i_s,j_s)$  concentric over $(i_t,j_t)$, put   $\sigma'=\sigma_{i_s\leftrightarrows i_t}.$
\end{itemize}
Then, ${\mathcal Z}_{\sigma'}\subset \overline{\mathcal Z}_\sigma$ and $\codim_{\ov{\mathcal Z}_\sigma}{\mathcal Z}_{\sigma'}=1$.
Conversely, each $\sigma'\in\bS_n^2(k)$ such that
${\mathcal Z}_{\sigma'}\subset \overline{\mathcal Z}_\sigma$
and $\codim_{\ov{\mathcal Z}_\sigma}{\mathcal Z}_{\sigma'}=1$
is obtained as described in {\rm (i)--(iv)}.
\end{lem}

For $\sigma\in\bS_n^2(k)$, put $N(\sigma)=\{\sigma'\in \bS_n^2(k):{\mathcal Z}_{\sigma'}\subset\ov{{\mathcal Z}_\sigma}\mbox{ and }\codim_{\ov{\mathcal Z}_\sigma}{\mathcal Z}_{\sigma'}=1\}$
and $P(\sigma)=\{\sigma'\in\bS_n^2(k):{\mathcal Z}_\sigma\subset\ov{{\mathcal Z}_{\sigma'}}\mbox{ and }
\codim_{\ov{\mathcal Z}_{\sigma'}}{\mathcal Z}_{\sigma}=1\}$.
Obviously, the number of elements in $N(\sigma)$ is equal to the sum of the number of concentric pairs,
the number of consecutive pairs
and the number of pairs of an arc with a next point.

\subsection{Computation of the number of $(d_0+1)$-dimensional orbits of ${\mathcal F}_u$}
Let $\sigma_0=(1,n-k+1)(2,n-k+2)\cdots(k,n)\in\mathbf{S}_n^2(k)$ be the element of Section \ref{minimal-orbit},
thus ${\mathcal Z}_{\sigma_0}$ is the unique minimal $Z(u)$-orbit of ${\mathcal F}_u$.
Let $d_0$ be its dimension.
We compute $|P(\sigma_0)|$, which coincides with
the number of $Z(u)$-orbits of ${\mathcal F}_u$
of dimension $d_0+1$.

\begin{lem} One has
$$|P(\sigma_0)|=\left\{\begin{array}{ll}
k&{\rm if}\ k={\frac n 2},\\
k+1&{\rm otherwise.}\\
\end{array}\right.
$$
\end{lem}
\begin{proof}
Note that for any $i=1,...,k-1$ the involution
$\sigma_i:=(\sigma_0)_{i\leftrightarrows i+1}=\cdots(i,n-k+i+1)(i+1,n-k+i)\cdots$
has exactly one concentric pair, namely
$\{(i,n-k+i+1),\ (i+1,n-k+i)\}$, and $\sigma_0=(\sigma_i)_{i\leftrightarrows i+1}$ so that, by Lemma \ref{lem-codim},
one has $\codim_{\ov{\mathcal Z}_{\sigma_i}}{\mathcal Z}_{\sigma_o}=1$. In such a way we get $k-1$ elements of $P(\sigma_0)$.

Further, if $k={\frac n 2}$, then the involution $\sigma':=(\sigma_0)_{k\leftrightarrows k+1}=
(1,k)\cdots(k+1,n)$ has exactly one consecutive pair, namely $\{(1,k),\ (k+1,n)\}$,
and $\sigma_0=\sigma'_{k\leftrightarrows k+1}$ so that, by Lemma \ref{lem-codim}, one has
$\codim_{\ov{\mathcal Z}_{\sigma'}}{\mathcal Z}_{\sigma_0}=1$.
Further note that, since involutions of $\bS_n^2(k)$ have no fixed points and an arc of $\sigma_0$ intersects any other arc,
these are the only possible elements of $P(\sigma_0)$, so that $|P(\sigma_0)|=k-1+1=k.$

If $k<{\frac n 2}$, then $\sigma_0$ has  $n-2k$ fixed points $k+1,\ldots,n-k$.
In that case, let $\sigma'=(\sigma_0)_{k\rar k+1}$ and $\sigma''=(\sigma_0)_{n-k+1\rar n-k}$.
Then $k$ is the next fixed point on the left of $(k+1,n)\in\sigma'$ and
$\sigma_0=\sigma'_{k+1\rar k}$ so that, by Lemma \ref{lem-codim},
one has $\codim_{\ov{\mathcal Z}_{\sigma'}}{\mathcal Z}_{\sigma_0}=1$.
Exactly in the same way, $n-k+1$ is the next fixed point on the right of $(1,n-k)\in\sigma''$
and $\sigma_0=\sigma''_{n-k\rar n-k+1}$ so that, by Lemma \ref{lem-codim}, one has
$\codim_{\ov{\mathcal Z}_{\sigma''}}{\mathcal Z}_{\sigma_0}=1$.
And again since an arc of $\sigma_0$ intersects any other arc,
these are the only possible elements of $P(\sigma_0)$, so that $|P(\sigma_0)|=k-1+2=k+1.$
\end{proof}

\subsection{The number of $1$-codimensional orbits in a singular component ${\mathcal K}^T$}
\label{proposition-singular-components}
Given $T$ standard,
recall that $P_{\sigma_T}$ is a link pattern without intersecting arcs and without fixed points under the arcs.
The component ${\mathcal K}^T$ is the closure of the $Z(u)$-orbit ${\mathcal Z}_{\sigma_T}$, hence
$|N(\sigma_T)|$ is the number of $1$-codimensional $Z(u)$-orbits lying in ${\mathcal K}^T$.
Our aim is to show

\begin{prop}
\label{prop-singular-components}
Let $T$ be standard. If one of the following holds:
\begin{itemize}
\item[\rm(i)] $|\tau^*(T)|=2$ and $1,n$ are fixed points of ${\sigma_T};$
\item[\rm(ii)] $|\tau^*(T)|=3$ and either $(1,n)\in\sigma_T$ or at least one point of $\{1,n\}$  is a fixed point of ${\sigma_T};$
\item[\rm(iii)]    $|\tau^*(T)|\geq 4;$
\end{itemize}
then $|N(\sigma_T)|> |P(\sigma_0)|.$
\end{prop}

\begin{proof}

\medskip
\noindent
{(i)} Let $\tau^*(T)=\{a_1,a_2\}$ where $a_1<a_2$.
Then,  $(a_1,a_1+1),(a_2,a_2+1)$
 are the only minimal arcs of the link pattern $P_{\sigma_T}$, as it was noted by Remark \ref{in512}.
Thus, any other arc is over $(a_1,a_1+1)$, or over $(a_2,a_2+1)$, or both.
Let $A_1=\{(a_1-k_1+1,a_1+k_1),\ldots,(a_1-1,a_1+2)\}$ be the set of arcs over $(a_1,a_1+1)$
but not over $(a_2,a_2+1)$. Let $A_2=\{(a_2-k_2+1,a_2+k_2),\ldots,(a_2-1,a_2+2)\}$ be the set of arcs over $(a_2,a_2+1)$
but not over $(a_1,a_1+1)$.
Finally, let $A_3=\{(a_1-k_1+1-k_3,a_2+k_2+k_3),\ldots,(a_1-k_1,a_2+k_2+1)\}$ be the set of arcs over both $(a_1,a_1+1)$ and $(a_2,a_2+1).$ Note that $|A_1|=k_1-1$, $|A_2|=k_2-1$ and $|A_3|=k_3$, so that $k_1+k_2+k_3=k$.
Note also that there are $k_1-1$ concentric pairs in $A_1\cup\{(a_1,a_1+1)\}$ and $k_2-1$ concentric pairs in $A_2\cup\{(a_2,a_2+1)\}$. If $A_3\ne\emptyset$, then there are $k_3-1$ concentric pairs in $A_3.$
\begin{itemize}
\item[a)] If $A_3=\emptyset$, then the concentric pairs in $A_1$ and $A_2$ provide $k-2$ elements of $N(\sigma_T)$ and
\begin{itemize}
 \item[(1)] If $a_1+k_1=a_2-k_2$, then $(a_1-k_1+1,a_1+k_1), (a_2-k_2+1,a_2+k_2)$ are consecutive arcs, which provides us one more element of $N(\sigma_T)$. Since $1$ and $n$ are fixed points, one has $a_1-k_1\geq 1$ and $a_1-k_1$ is the next point on the left of both $(a_1-k_1+1,a_1+k_1)$ and $(a_2-k_2+1,a_2+k_2)$ (resp. $a_2+k_2+1\leq n$, and $a_2+k_2+1$ is the next point on the right of both $(a_1-k_1+1,a_1+k_1)$ and $(a_2-k_2+1,a_2+k_2)$), providing us altogether four more elements of $N(\sigma_T).$ We get in this case $|N(\sigma_T)|\geq k-2+5=k+3$.
\item[(2)] If $a_1+k_1<a_2-k_2$, then $a_1+k_1+1$ is the next point on the right of
$( a_1-k_1+1,a_1+k_1)$ and $a_2-k_2$ is the next point on the left of $(a_2-k_2+1,a_2+k_2)$, providing us two new elements of $N(\sigma_T)$. Further, exactly as in (1), since $1$ and $n$ are fixed points one has $a_1-k_1\geq 1$ and it is the next point on the left of $(a_1-k_1+1,a_1+k_1)$ (resp. $a_2+k_2+1\leq n$ and it is the next point on the right of  $(a_2-k_2+1,a_2+k_2)$), giving us two more elements of $N(\sigma_T).$ We get in this case
$|N(\sigma_T)|\geq k-2+4=k+2.$
\end{itemize}
\item[b)] If $A_3\ne\emptyset$, then the concentric pairs in $A_1,\ A_2,\ A_3$ give $k-3$ elements of $N(\sigma_T)$.  Also $a_1+k_1=a_2-k_2$ (since there are no fixed points under the arcs and no intersecting arcs), so that $(a_1-k_1+1,a_1+k_1),(a_2-k_2+1,a_2+k_2)$ are consecutive arcs, which gives us a new element of $N(\sigma_T).$
As well, $(a_1-k_1,a_2+k_2+1)$ is concentric with both $(a_1-k_1+1,a_1+k_1)$ and $(a_2-k_2+1,a_2+k_2)$, providing us two more elements of $N(\sigma_T)$.  Finally, exactly as in (a), since $1$ and $n$ are fixed points  we get that $a_1-k_1-k_3\geq 1$ is the next point on the left of $(a_1-k_1-k_3+1,a_2+k_2+k_3)$ and $a_2+k_2+k_3+1\leq n$ is the next  point on the right of
the same arc, giving us two more elements of $N(\sigma_T).$
We get $|N(\sigma_T)|\geq k-3+5=k+2$.
\end{itemize}
So in any case,  $|N(\sigma_0)|\geq k+2> |P(\sigma_0)|$.

\smallskip
To prove the rest of the proposition (cases (ii) and (iii)), we reason by induction on the pair $(|\tau^*(T)|,n)$, starting with the case $|\tau^*(T)|=3$ and $n=7$. This is the minimal case for situation (ii) and in this case $(1,7)\not\in\sigma_T$
(since there are no fixed points under the arcs). By \ref{observation-3}, it is enough to consider the case of $1$ being a fixed point. Since $|\tau^*(T)|=3$
we get that
\[T=\young(13,25,47,6)\]
so that $\sigma_T=(2,3)(4,5)(6,7)$ and $|N(\sigma_T)|=6>4.$

Now assume the statement is true for $n-1$ and show for $n.$ Let $T'=T_{n-1}.$
\begin{itemize}
\item[a)]
If $1,n$ are fixed points of $\sigma_T$, then in particular $c_T(n)=1$ so that $T'$ has columns of lengths $(n-k-1,k)$ and $n-1\notin\tau^*(T)$,
hence $|\tau^*(T')|=|\tau^*(T)|$. Moreover, since $|P^0_{\sigma_T}|\geq 2$ one has $k<\frac{n-1}{2}$.

\begin{itemize}
\item If $|\tau^*(T)|=3$, then $|\tau^*(T')|=3$ and $1$ is a fixed point of $\sigma_{T'}$.
\item If $|\tau^*(T)|\geq 4$, then $|\tau^*(T')|\geq 4$.
\end{itemize}
In both cases, the induction hypothesis holds and gives $|N(\sigma_{T'})|>k+1$. \\
There is a natural injective map $\mathbf{S}_{n-1}^2(k)\rightarrow \mathbf{S}_{n}^2(k)$
where $\sigma\in \mathbf{S}_{n-1}^2(k)$ is regarded as an involution in $\mathbf{S}_{n}^2(k)$ satisfying $\sigma(n)=n$. In that way, one has $\sigma_{T'}=\sigma_T$ and
we can regard $\sigma\in N(\sigma_{T'})$ as $\sigma\in N(\sigma_{T})$.
Therefore $|N(\sigma_{T})|\geq |N(\sigma_{T'})|>k+1=|P(\sigma_0)|$.

\item[b)]
If $1$ or $n$ is an end point of $\sigma_T$, then by \ref{observation-3} we can assume
that $n$ is an end point, hence $c_T(n)=2.$
The tableau $T'$ has columns of lengths $(n-k,k-1)$ where $k-1<n-k$. One has
$$|\tau^*(T')|=\left\{\begin{array}{ll}|\tau^*(T)|-1&{\rm if\ }c_T(n-1)=1,\\
 |\tau^*(T)|&{\rm otherwise.}\\
 \end{array}\right.$$
\begin{itemize}
\item If $|\tau^*(T)|=3$, then, by hypothesis, $1$ is a fixed point of $\sigma_{T'}$.
If $c_T(n-1)=1$, then $|\tau^*(T')|=2$ and $n-1$ is also a fixed point of $\sigma_{T'}$, thus, by (i), $|N(\sigma_{T'})|>k$.
If $c_T(n-1)=2$, then $|\tau^*(T')|=3$, thus by induction hypothesis again $|N(\sigma_{T'})|>k$.
\item Suppose $|\tau^*(T)|\geq 4$. If $c_T(n-1)=1$, then $|\tau^*(T')|\geq 3$ and $n-1$ is a fixed point of $\sigma_{T'}$ so that by induction hypothesis $|N(\sigma_{T'})|>k$. If $c_T(n-1)=2$, then
$|\tau^*(T')|\geq 4$, thus again by induction hypothesis $|N(\sigma_{T'})|>k$.
\end{itemize}
Let us show that there exists an injective map $\phi:\bS_{n-1}^2(k-1)\rar \bS_n^2(k)$ such that for $\sigma\in N(\sigma_{T'})$ one has $\phi(\sigma)\in N(\sigma_T).$
Since $k-1<n-k$, every $\sigma\in \mathbf{S}_{n-1}^2(k-1)$ has at least one fixed point.
For $\sigma\in \mathbf{S}_{n-1}^2(k-1)$, let $p_\sigma$ be its maximal fixed point.
Put $\phi(\sigma)=(p_\sigma,n)\cdot\sigma\in\mathbf{S}_{n}^2(k)$.
Note that $\phi$ is injective and that
 there are no fixed points under the arc $(p_\sigma,n)\in\phi(\sigma)$. Note also that
$(i_t,j_t)\in\sigma$ is a bridge over $p_\sigma$  if and only if $(i_t,j_t)$ and $(p_\sigma,n)$ have a crossing in $\phi(\sigma)$. Thus $b(\sigma)+c(\sigma)=b(\phi(\sigma))+c(\phi(\sigma))$ (see \ref{definition-link-pattern}) so that, by Proposition \ref{proposition-orbits},
$$\codim_{\F_u} \Z_{\phi(\sigma)}=\codim_{\F_{u'}}\Z_\sigma\,,\eqno{(*)}$$
where $Y(u')$ has two columns of lengths $(n-k,k-1)$.
 Note also that $\phi(\sigma_{T'})=\sigma_{T}$. Put $q=p_{\sigma_{T'}}$.

Consider $\sigma\in N(\sigma_{T'})$, and let us show $\phi(\sigma)\in N(\sigma_T)$. By $(*)$ it is enough to show
that $\phi(\sigma)\preceq\sigma_T$.

\begin{itemize}
\item
If $p_\sigma\leq q$, then  $\phi(\sigma)\preceq\sigma_T$ by definition of the order $\preceq$.
\item
If $p_\sigma >q$, then $p_\sigma$ is an end point of $\sigma_{T'}$. Thus, by Lemma \ref{lem-codim} and definition of $N(\sigma_{T'})$,
one has that $q$ is the next fixed point on the left of $(p_\sigma,j)\in \sigma_{T'}$ for some $j>p_\sigma$, and $\sigma=(\sigma_{T'})_{p_\sigma\rightarrow q}$. It follows
that $(q,n)$ is concentric over $(p_\sigma,j)$ in $\sigma_T$ and
$\phi(\sigma)=(\sigma_{T})_{p_\sigma\rightleftarrows q}$,
thus $\phi(\sigma)\prec\sigma_T$.
\end{itemize}
Thus, $|N(\sigma_T)|\geq|N(\sigma_{T'})|>k$.

If $k=\frac{n}{2}$, then we get $|N(\sigma_T)|>|P(\sigma_0)|$ and the proof is complete.

Let us suppose now $k<\frac{n}{2}$. Then $\sigma_T$ has at least one fixed point. Let $p$ be the maximal fixed
point of $\sigma_T$. Note that it is the next point on the left of $(q,n)\in\sigma_T$, so that
$\sigma:=(\sigma_T)_{q\rightarrow p}\in N(\sigma_T)$. Since $(p,n)\in\sigma$ is a bridge over $q$, one has
$\sigma\notin \phi(S_{n-1}^2(k-1))$, and in particular,
 $\sigma\notin\phi(N(\sigma_{T'}))$.
Thus, in this case also,  $|N(\sigma_T)|>k+1=|P(\sigma_0)|$.
\end{itemize}
The proof is now complete.
\end{proof}

\subsection{Proof of Theorems \ref{first-crit}, \ref{second-crit}, \ref{third-crit}}
\label{proof-third-criterion}
Theorems \ref{first-crit} and \ref{second-crit} both follow by combining
Propositions \ref{prop-Betti-orbits}, \ref{prop-nonsingular-components}, \ref{prop-singular-components}.
To prove Theorem \ref{third-crit}, we show that $|N(\sigma_T)|$ coincides with the
number $\eta(T)$ of components ${\mathcal K}'\subset{\mathcal F}_u$ such that $\mathrm{codim}_{{\mathcal F}_u}{\mathcal K}^T\cap {\mathcal K'}=1$.
We claim that each $1$-codimensional
$Z(u)$-orbit of ${\mathcal F}_u$ is contained in exactly two components.
Then, whenever $\sigma\in N(\sigma_T)$ there is a unique
component ${\mathcal K}'_\sigma\subset {\mathcal F}_u$
such that ${\mathcal Z}_\sigma\subset {\mathcal K}^{T}\cap{\mathcal K}'_\sigma$.
By \cite[Theorem 4.1]{M-P3}, if ${\mathcal K}^T\cap {\mathcal K'}$ has codimension 1,
it is irreducible and, in particular, it contains exactly one $Z(u)$-orbit of maximal dimension.
We derive that the map $\sigma\mapsto {\mathcal K}'_\sigma$ is
bijective from $N(\sigma_T)$ onto the set of components ${\mathcal K}'\subset{\mathcal F}_u$ such that $\mathrm{codim}_{{\mathcal F}_u}{\mathcal K}^T\cap {\mathcal K'}=1$.
Therefore, we get $|N(\sigma_T)|=\eta(T)$.

Thus, it remains to show our claim that
for $\sigma\in \mathbf{S}_n^2(k)$ such that
$\mathrm{codim}_{{\mathcal F}_u}{\mathcal Z}_\sigma=1$, one has $|P(\sigma)|=2$. Note that, by
Proposition \ref{proposition-orbits},
$\sigma$ either has a unique crossing of two arcs or has a unique fixed point $p$ with a bridge $(i,j)$ over it.

\begin{itemize}
\item[a)] If $\sigma$ has a crossing, i.e. two arcs $(i,j),(i',j')$ such that $i<i'<j<j'$,
then $(i,i'),(j,j')\in\sigma_{i'\rightleftarrows j}$ are consecutive and
$(i,j'),(i',j)\in \sigma_{i\rightleftarrows i'}$ are concentric, and due to Lemma \ref{lem-codim}, $\sigma_{i'\rightleftarrows j},\  \sigma_{i\rightleftarrows i'}$ are the only elements of $P(\sigma)$.
\item[b)] If $\sigma$ has a fixed point $p$ under an arc $(i,j)$,
then $i$ is the next point on the left of $(p,j)$ in $\sigma_{i\rightarrow p}$
and $j$ is the next point on the right of $(i,p)$ in $\sigma_{j\rightarrow p}$.
Due to Lemma \ref{lem-codim},
we get $P(\sigma)=\{\sigma_{i\rightarrow p},\sigma_{j\rightarrow p}\}.$

\end{itemize}
The proof of Theorem \ref{third-crit} is then complete.
\hfill {\large $\Box$}

\section{Further aspects on the singularity of the components}\label{further-speculations}

In this final section, we discuss the singularity of the irreducible components of the Springer fiber ${\mathcal F}_u$,
without assuming anymore that $Y(u)$ has two columns.

\subsection{Description of Springer fibers having singular components.}
\label{classification}
As it was shown in \cite{F-M}, the variety ${\mathcal F}_u$ admits singular components if and only if
the Young diagram $Y(u)$ contains as a subdiagram one of the following two diagrams:
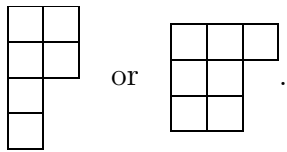
\[\yng(2,2,1,1)\quad\mbox{or}\quad\yng(3,2,2).\]

Nevertheless, up to now, the characterization of the components which are indeed singular is
limited to the two-column case.

In the light of the previous sections, specifically Theorem \ref{second-crit},
it is natural to ask whether a singular component can always be characterized
by the fact that its distribution of Betti numbers is non-symmetric,
implying that it does not satisfy the Poincar\'e duality.

\subsection{Poincar\'e polynomial for the simplest singular component outside two-column case}
\label{example322}
The first example of a singular component apart from the two-column case
is the component ${\mathcal K}^T$ corresponding to the tableau
\[T=\young(125,34,67)\]
(see \cite[\S 2.3]{F-M}).
It lies in the Springer fiber ${\mathcal F}_u$ attached to
a nilpotent $u\in\mathrm{End}(\mathbb{C}^7)$ with three Jordan blocks of sizes
$3,2,2$.
Fix a Jordan basis $(e_1,\ldots,e_7)$ of  $\mathbb{C}^7$ with
$u$ acting by $e_1\mapsto e_4\mapsto e_7\mapsto 0$,
$e_2\mapsto e_5\mapsto 0$, $e_3\mapsto e_6\mapsto 0$,
and let $B\subset GL(\mathbb{C}^7)$ be the Borel subgroup of lower triangular isomorphisms
for this basis. We have computed all the intersections between
the component ${\mathcal K}^T$ and the $B$-orbits of the variety of complete flags of $\mathbb{C}^7$,
and we have obtained that each nonempty intersection of ${\mathcal K}^T$ with a $B$-orbit is in fact isomorphic as
an algebraic variety to a cell, i.e. an affine space $\mathbb{C}^{d}$ with $d\in\{0,\ldots,6\}$.
These intersections then form a cell decomposition of ${\mathcal K}^T$,
from which, leading our computations to the end,
we have obtained the following expression for the Poincar\'e polynomial:
\[P(x):=\sum_{m=0}^6\dim H^{2m}({\mathcal K}^T,\mathbb{Q})\,x^m=1+6x+18x^2+28x^3+22x^4+8x^5+x^6.\]
Thus,
the component has a non-symmetric distribution of Betti numbers and then the Poincar\'e duality fails.

\subsection{Inductive construction of components not satisfying Poincar\'e duality}
To go further, let us first point out a property of inductive transmitting of Poincar\'e duality.
Let $T$ be a standard tableau of shape $Y(u)$, and correspondingly we have a component ${\mathcal K}^T\subset {\mathcal F}_u$.
Let $T'=T_{n-1}$ and let $Y'=Y_{n-1}^T$ be its shape.
Consider a nilpotent endomorphism $u'$ with $Y'=Y(u')$.
Then, corresponding to $T'$, we have an irreducible component
${\mathcal K}^{T'}\subset {\mathcal F}_{u'}$.
We have the following

\begin{prop}
\label{fiber-bundle}
Assume that $n$ is in the last column of $T$, and let $m$ be the length of the last column of $T$.
Then ${\mathcal K}^T$ is naturally a locally trivial fiber bundle over $\mathbb{P}^{m-1}$,
of fiber ${\mathcal K}^{T'}$.
In particular, the Poincar\'e polynomial of ${\mathcal K}^T$ is palindromic if and only if
the one of ${\mathcal K}^{T'}$ is.
\end{prop}

\begin{proof}
Let $s$ be the index of the last column of $T$ and let ${\mathcal H}$ be the variety of hyperplanes
$H\subset V$ such that $H\supset \ker u^{s-1}$.
This is a projective variety which is naturally isomorphic to the variety of hyperplanes
of the space $V/\ker u^{s-1}$, hence ${\mathcal H}\cong \mathbb{P}^{m-1}$.
Recall that the component ${\mathcal K}^T$ is obtained as the Zariski closure of the subset ${\mathcal F}_u^T\subset {\mathcal F}_u$
(see Section \ref{parameterization-components}). Notice that every flag $F=(V_0,\ldots,V_n)\in{\mathcal F}_u^T$ satisfies
$V_{n-1}\supset \ker u^{s-1}$, hence $V_{n-1}\in{\mathcal H}$ whenever $F\in{\mathcal K}^T$.
Then, the fact that the map
\[\Phi:{\mathcal K}^T\rightarrow {\mathcal H},\ (V_0,\ldots,V_n)\mapsto V_{n-1}\]
is indeed a locally trivial fiber bundle of fiber isomorphic to the
component ${\mathcal K}^{T'}$ follows from the proof of \cite[Theorem 2.1]{F-M}.

The Poincar\'e polynomial of the component ${\mathcal K}^T$ is now obtained as the product of the Poincar\'e polynomial of
${\mathcal K}^{T'}$ and the one of $\mathbb{P}^{m-1}$, which is palindromic. This implies the last statement of the proposition.
\end{proof}

\medskip

Then, combining Theorem \ref{second-crit},
Sections \ref{classification}--\ref{example322} and the above proposition, we obtain the following

\begin{thm}
Whenever the Springer fiber ${\mathcal F}_u$ has a singular component,
it admits a component whose Poincar\'e polynomial is not palindromic.
\end{thm}

\begin{proof}
We prove the theorem by induction on $n$, the number of boxes in $Y(u)$.
As an initialization, we suppose that $Y(u)$ has two columns, or three rows of lengths $(3,2,2)$,
then the property follows
either from Theorem \ref{second-crit} or from Section \ref{example322}, respectively.
Now assume that $Y(u)$ is neither of two-column nor of $(3,2,2)$ type.
Let $Y'$ be the subdiagram obtained from $Y(u)$ by deleting the last box of the last column,
and take $u'\in\mathrm{End}(\mathbb{C}^{n-1})$ nilpotent such that $Y(u')=Y'$.
Then, by the description in Section \ref{classification}, the Springer fiber ${\mathcal F}_{u'}$
also admits a singular component. By induction hypothesis, there is a standard tableau $T'$ of shape $Y'$
whose corresponding component ${\mathcal K}^{T'}\subset {\mathcal F}_{u'}$ has a non-palindromic Poincar\'e polynomial.
Let $T$ be the tableau of shape $Y(u)$ such that $T_{n-1}=T'$. By Proposition \ref{fiber-bundle},
the Poincar\'e polynomial of ${\mathcal K}^T\subset {\mathcal F}_u$ is also non-palindromic.
\end{proof}

\bigskip
 {\bf Acknowledgements.}  We would like to thank T.A. Springer for the inspiring correspondence.

\section*{Index of notation}

\noindent
\ \ref{introduction-1}\quad $V$, $n$, $u$, ${\mathcal F}$, ${\mathcal F}_u$

\noindent
\ \ref{parameterization-components}\quad $Y(u)$, $T_i$, ${\mathcal F}_u^T$, ${\mathcal K}^T$, $k$

\noindent
\ \ref{previous-criterion}\quad $\tau_0$, $F_\tau$, $X(\tau_0)$

\noindent
\ \ref{first-criterion}\quad $\mathbf{S}_n$, $\mathbf{S}_n^2$, $\mathbf{S}_n^2(k)$, $(i,j)\in\sigma$, $\sigma_T$, $c_T(i)$, $\tau^*(T)$

\noindent
\ \ref{second-criterion}\quad $H^m(X,\mathbb{Q})$

\noindent
\ \ref{plan}\quad $Z(u)$, $\overline{k}$

\noindent
\ \ref{definition-link-pattern}\quad $P_\sigma$, $P_\sigma^+$, $P_\sigma^-$, $P_\sigma^0$, $c(\sigma)$, $b(\sigma)$, $(i,j)\in[a,b]$,
$R_{a,b}(\sigma)$, $\sigma'\preceq \sigma$

\noindent
\ \ref{sect-proposition-orbits}\quad ${\mathcal Z}_\sigma$

\noindent
\ \ref{minimal-orbit}\quad $\sigma_0$, $\sigma_0(k)$

\noindent
\ \ref{lem-cell-decomposition}\quad $n_\mathrm{inv}(w)$

\noindent
\ \ref{proposition-Betti-orbits}\quad $d_0$, $n(T:m)$

\noindent
\ \ref{notation-link-patterns}\quad $\sigma_{i\rightarrow p}$, $\sigma_{i\leftrightarrows j}$, $N(\sigma)$, $P(\sigma)$

\end{document}